\documentclass[10pt,reqno]{amsart}

\usepackage[english]{babel}
\usepackage{amsmath,amssymb,amscd}
\usepackage{shuffle}
\usepackage[all]{xy}
\usepackage{todonotes}
\usepackage[normalem]{ulem}

\newtheorem{definition}{Definition}[section]

\newcommand{\Lie}{\mathrm{Lie}}

\newcommand{\ls}{{\mathfrak{ls}}}
\newcommand{\pls}{{\mathfrak{pls}}}
\newcommand{\ff}{{\mathfrak{f}}}

\newcommand{\Dari}{{\mathrm{Dari}}}
\newcommand{\Darit}{{\mathrm{Darit}}}

\newcommand{\Z}{{\mathbb{Z}}}
\newcommand{\C}{{\mathbb{C}}}
\newcommand{\R}{{\mathbb{R}}}
\newcommand{\Q}{{\mathbb{Q}}}
\newcommand{\N}{{\mathbb{N}}}

\newcommand{\Ac}{{\mathcal{A}}}

\newcommand{\Gc}{{\mathcal{G}}}

\newcommand{\cA}{{\mathcal{A}}}
\newcommand{\cB}{{\mathcal{B}}}
\newcommand{\cC}{{\mathcal{C}}}

\newcommand{\cMZ}{{\mathcal{MZ}}}
\newcommand{\cNMZ}{{\mathcal{NMZ}}}
\newcommand{\cO}{{\mathcal{O}}}

\newcommand{\Ec}{{\mathcal{E}}}

\newcommand{\Bc}{{\mathcal{B}}}

\newcommand{\Zc}{{\mathcal{Z}}}

\newcommand{\weps}{\widetilde{\eps}}

\def\eps{{\varepsilon}}

\def\ff{{\mathfrak f}}

\def\fu{{\mathfrak u}}
\def\fv{{\mathfrak v}}

\def\fH{{\mathfrak H}}

\def\uk{{\underline{k}}}

\def\grt{{\mathfrak {grt}}}
\def\ds{{\mathfrak {ds}}}
\def\nz{{\mathfrak {nz}}}

\def\nmz{{\mathfrak {nmz}}}

\def\dd{{\mathrm {d}}}

\DeclareMathOperator{\Hom}{Hom}
\DeclareMathOperator{\SL}{SL}
\DeclareMathOperator{\Der}{Der}
\DeclareMathOperator{\ad}{ad}
\DeclareMathOperator{\Frac}{Frac}
\DeclareMathOperator{\cI}{\mathcal{I}}
\DeclareMathOperator{\U}{\mathcal{U}}

\theoremstyle{definition}
\newtheorem{dfn}{Definition}[section]
\newtheorem{rmk}[dfn]{Remark}

\theoremstyle{plain}
\newtheorem{prop}[dfn]{Proposition}

\newtheorem{thm}[dfn]{Theorem}
\newtheorem{lem}[dfn]{Lemma}
\newtheorem{coro}[dfn]{Corollary}
\newtheorem*{intcoro}{Corollary}
\newtheorem*{intthm}{Theorem}

\numberwithin{equation}{section}

\begin{document}

\title[Elliptic multizetas]{Elliptic multizetas and the elliptic double shuffle relations}

\author{Pierre Lochak}
\address{P.L.:  CNRS and Institut Math\'ematique de Jussieu, 
Universit\'e P. et M. Curie, 4 place Jussieu, F-75252 Paris Cedex 05\newline 
E-mail address: {\tt  pierre.lochak@imj-prg.fr}}

\author{Nils Matthes}
\address{N.M.: Fachbereich Mathematik (AZ)
Universit\"at Hamburg, Bundesstrasse 55, D-20146 Hamburg  \newline 
E-mail address: {\tt  nils.matthes@uni-hamburg.de}}

\address{Current address: Mathematical Institute, University of Oxford, Andrew Wiles Building, Radcliffe Observatory Quarter, Woodstock Road, Oxford OX26GG, United Kingdom \newline 
E-mail address: {\tt nils.matthes@maths.ox.ac.uk}}

\author{Leila Schneps}
\address{L.S.:  CNRS and Institut Math\'ematique de Jussieu, 
Universit\'e P. et M. Curie, 4 place Jussieu, F-75252 Paris Cedex 05\newline 
E-mail address: {\tt  leila.schneps@imj-prg.fr}}

\subjclass[2010]{11M32}
\keywords{Multiple zeta values, Grothendieck--Teichm\"uller theory, Moulds}

\begin{abstract}

We define an {\it elliptic generating series} whose coefficients, the
{\it elliptic multizetas}, are related to the elliptic analogues of 
multiple zeta values introduced by Enriquez as the coefficients of his
elliptic associator; both sets of coefficients lie in $\mathcal{O}(\fH)$, 
the ring of functions on the Poincar\'e upper half-plane $\fH$.
The elliptic multizetas generate a $\Q$-algebra $\Ec$ which is an 
elliptic analogue of the algebra of multiple zeta values.  
Working modulo $2\pi i$, we show that the algebra $\Ec$ 
decomposes into a geometric and an arithmetic part and study the precise
relationship between the elliptic generating series and the elliptic associator
defined by Enriquez.  We show that the elliptic 
multizetas satisfy a double shuffle type family of algebraic relations 
similar to the double shuffle relations satisfied by multiple zeta values.  
We prove that 
these elliptic double shuffle relations give all algebraic relations among elliptic 
multizetas if (a) the classical double shuffle relations give all algebraic 
relations among multiple zeta values and (b) the elliptic double shuffle Lie algebra has 
a certain natural semi-direct product structure analogous to that established by 
Enriquez for the elliptic Grothendieck-Teichm\"uller Lie algebra. 
\end{abstract}

\maketitle

\section{Introduction}
\label{sec:introduction}

\vspace{.3cm}
\subsection{Elliptic multizetas}
An elliptic analogue of the multiple zeta values first made an explicit appearance in 
Enriquez' article \cite{Enriquez:Emzv} under the name ``analogues elliptiques 
des nombres multizetas''.  They arise as coefficients of his elliptic 
associator constructed in \cite{Enriquez:EllAss}, which is closely related to 
the elliptic Knizhnik--Zamolodchikov--Bernard (KZB) equation 
\cite{CEE,LR} and to multiple elliptic polylogarithms \cite{BL,LR}; 
more recently, they have even found applications to computations in high energy 
physics \cite{BMMS}.  Taking the regularized limit $\tau \to i\infty$ of 
elliptic multizetas, one retrieves the classical multiple zeta values
\cite{Enriquez:Emzv,Matthes:Thesis}, which gives the 
explicit connection between the genus zero and genus one multizetas.
The idea of considering the graded $\Q$-algebra generated by these coefficients,
was introduced in \cite{BMS,Matthes:Thesis,Matthes:Edzv}, which provide some
explicit dimension results in depth 2.

Recall that the Drinfel'd associator $\Phi_{\rm KZ}$, first introduced in
\cite{Drinfeld}, is a power series in two
non-commutative variables\footnote{Throughout this paper we use the 
same definitions and conventions as \cite{Enriquez:EllAss}; in particular
$\Phi_{\rm KZ}$ is defined in \S 5.1.},
which is the generating series for the usual 
multiple zeta values, by the work of Le and Murakami \cite{LM}.  In analogy with this, Enriquez's elliptic associator,
which is defined as a pair of monodromies (cf.~\S 5.2 of 
\cite{Enriquez:EllAss}),
takes the form of a pair of group-like power series in two 
non-commutative variables $a$ and $b$:
$$A(\tau),B(\tau)\in \mathcal{O}(\mathfrak{H})\langle\!\langle a,b\rangle\!\rangle,$$
where $\mathcal{O}(\mathfrak{H})$, 
denotes the ring of holomorphic functions 
of one variable $\tau$ running through the Poincar\'e upper 
half-plane.  We call the coefficients of $A(\tau)$ and 
$B(\tau)$ {\it $A$-elliptic multizetas} and {\it $B$-elliptic multizetas}, 
or $A$-EMZs and $B$-EMZs.  The acronym EMZ stands for
{\it elliptic multizetas}; since they are functions of $\tau$ and not complex numbers, 
we drop the word ``values'' in the elliptic situation\footnote{These values, 
which are Enriquez' ``elliptic analogues of MZVs'', and the E-EMZs introduced
below, are very different from Brown's ``multiple modular values'' 
\cite{Brown:MMV}, which are complex numbers.}. 

\vspace{.2cm}
The main new object introduced in this article is a third power series
$$C(\tau)\in \mathcal{O}(\mathfrak{H})\langle\!\langle a,b\rangle\!\rangle$$ 
and its logarithm $E(\tau)={\rm log}\,C(\tau)$ (\S 3.1).  
The series $E(\tau)$ is
called the {\it elliptic generating series}, and its coefficients, in
$\mathcal{O}(\mathfrak{H})$, are called {\it $E$-elliptic multizetas} or 
$E$-EMZs.  We write $\mathcal{E}$, $\mathcal{A}$ and $\mathcal{B}$ for the 
vector spaces generated by the coefficients of $C(\tau)$, the $A$-EMZs 
and the $B$-EMZs respectively, which are by definition subspaces of 
$\mathcal{O}(\mathfrak{H})$.
We show in Lemma \ref{2piiEAB} that like $A(\tau)$ and $B(\tau)$, $C(\tau)$ is a
group-like power series, which implies that each of the three subspaces $\Ec$, $\Ac$ and $\Bc$ of $\mathcal{O}(\mathfrak{H})$ actually forms a $\Q$-algebra. 
In Lemma \ref{sameoldthing}, we show that the coefficients of $E(\tau)$,
the $E$-EMZs, form a system of algebra generators for $\mathcal{E}$.  This
is the generating system we will study in the rest of the article.

One of the main results in \cite{Enriquez:EllAss} concerning the power 
series $A(\tau)$, $B(\tau)$ is that they can be written in the 
form\footnote{Throughout this article we use the dot notation $\cdot$ to 
indicate the action of automorphisms or derivations on elements.}
$g(\tau)\cdot A$ and $g(\tau)\cdot B$, where $g(\tau)$ is
an automorphism of $\mathcal{O}(\mathfrak{H})\langle\!\langle a,b\rangle
\!\rangle$ (introduced in \S 5.1 of \cite{Enriquez:EllAss} but recalled
in \S 2 below), and $A$ and $B$ are power series in 
$\Zc[2\pi i]\langle\!\langle a,b\rangle\!\rangle$ (defined in \S 3.5 of 
\cite{Enriquez:EllAss} with the notation $A_+$, $A_-$, but recalled
in \S 3.1 below).  This property of the elliptic associator is the motivation 
for our definition of the power series $C(\tau)$ directly in the form
$g(\tau)\cdot C$, where $C$ is a group-like power series in 
$\Zc[2\pi i]\langle\!\langle a,b\rangle\!\rangle$ closely related to $A$ and $B$ (\S 3.1). Since $g(\tau)$ is an 
automorphism, the power series $E(\tau)= \log\,C(\tau)$ then naturally takes 
the form $g(\tau)\cdot E$ with $E=\log\,C\in \Zc[2\pi i]\langle\!\langle
a,b\rangle\!\rangle$.

We let $\Ec^{\rm geom}$ denote the $\Q$-algebra generated by the
coefficients of $g(\tau)$ (in a precise sense explained 
in \S 2). These coefficients lie in $\mathcal{O}(\mathfrak{H})$, and 
are realized as particular linear combinations of iterated integrals of 
Eisenstein series for $\SL_2(\Z)$ (see \cite{Brown:MMV,Manin:Iterated}). 
We note that for any ring $R$, $g(\tau)$ induces an automorphism of
$\bigl(\Ec^{\rm geom}\otimes R\bigr)\langle\!\langle a,b\rangle\!\rangle$.
We use this fact frequently below.

The structure of the $\Q$-algebra $\Ec^{\rm geom}$ is the main topic of \S 2.
It is related to the bigraded Lie algebra $\fu^{\rm geom}$ of the prounipotent 
radical of $\pi_1^{\rm geom}(MEM)$, where $MEM$ denotes the Tannakian category 
of universal mixed elliptic motives \cite{HM}.  More precisely, $\Ec^{\rm geom}$
is related to the Lie algebra $\fu$ which is the (completion of the) image of 
$\fu^{\rm geom}$ under the monodromy representation from $\fu^{\rm geom}$ 
to the Lie algebra of derivations of a free Lie algebra on two generators 
whose existence is shown in \cite{HM},~\S 22. The explicit generators of $\fu$ are well known, cf. Definition \ref{epses}. 

The first results of this article are summarized in the following theorem,
whose proof rests in large part on the $\C$-linear independence
of iterated integrals of Eisenstein series proved in
Theorem \ref{matthes} and its corollary \ref{cor:Eis}.\footnote{The second
author subsequently generalized this result to the case of arbitrary
quasimodular forms for ${\rm SL}_2(\Z)$ \cite{Matthes:IterIntQMod}.}

\begin{intthm}\label{introthm} 
(i) [Thm. \ref{duality}]
There is a natural isomorphism
\begin{equation*}
\Ec^{\rm geom}\cong {\U}(\fu)^\vee,
\end{equation*}
where $\U(\fu)^{\vee}$ is the graded dual\footnote{We refer to Section 1.5.1 for our (slightly non-standard) conventions concerning graded duals.} of the universal enveloping algebra 
of the Lie algebra $\fu$. In particular, $\Ec^{\rm geom}$ is a commutative, graded, Hopf  $\Q$-algebra.

\vspace{.2cm}
(ii) [Cor. \ref{indepcor}] The subalgebra of 
$\mathcal{O}(\mathfrak{H})$ generated by $\Ec^{\rm geom}$ and $\Zc[2\pi i]$ is 
isomorphic to the tensor product 
\begin{equation}\label{tprod}
\Ec^{\rm geom}\otimes_\Q \Zc[2\pi i].
\end{equation}
\end{intthm}
\vskip .2cm
\begin{intcoro} The $\Q$-algebras $\Ac$, $\Bc$ and $\Ec$ are isomorphic to subalgebras of the
tensor product \eqref{tprod}.
\end{intcoro}
\begin{proof} Since $A(\tau)$ has the form $g(\tau)\cdot A$, the
coefficients of $A(\tau)$ are algebraic expressions in elements of
$\Ec^{\rm geom}$ and $\Zc[2\pi i]$. The same holds for $B(\tau)$ and
$C(\tau)$.  The result then follows from (ii) of the Theorem.
\end{proof}

\subsection{Structure of the elliptic multizeta algebras mod $2\pi i$}
\vspace{.2cm}
For technical reasons linked to our use of Grothendieck-Teichm\"uller
theory and \'Ecalle's mould theory, we restrict our study of the three
different types of elliptic multizetas to objects to their reductions 
modulo $2\pi i$ in the following sense.

Let $\overline{\Zc}$ denote the quotient of $\Zc[2\pi i]$ by the ideal 
generated by $2\pi i$, which is isomorphic to the quotient of the algebra 
of multiple zeta values $\Zc$ by the ideal generated
by $\zeta(2)=-\frac{(2\pi i)^2}{24}$.  The quotient of
$\Ec^{\rm geom}\otimes_\Q \Zc[2\pi i]$ by the ideal generated by $1\otimes 2\pi i$ is
isomorphic to 
$$\Ec^{\rm geom}\otimes_\Q \overline{\Zc}.$$

\begin{definition} Let $\overline{\Ec}$ (resp.~$\overline{\Bc}$) denote
the image of the subalgebra $\Ec$ (resp.~$\Bc$) of 
$\Ec^{\rm geom}\otimes \Zc[2\pi i]$ in the
quotient $\Ec^{\rm geom}\otimes \overline{\Zc}$.  
Let the reduced power series $\overline{E}(\tau)$ (resp.~$\overline{B}(\tau)$) be obtained from $E(\tau)$ (resp.~$B(\tau)$) by reducing the coefficients
 from $\Ec$ to $\overline\Ec$ (resp.~from $\Bc$ to $\overline\Bc$).    
\end{definition}

The case of $\Ac$ is slightly different, because it follows from
the definition of $A(\tau)$ given in (\ref{AandB}) and (\ref{uffa}) that
the ring $\Ac$ lies in $\Q\cdot 1+\Ec^{\rm geom}\otimes
2\pi i\Zc[2\pi i]$, and therefore the image of this ring in the quotient
$\Ec^{\rm geom}\otimes \overline\Zc$ is just $\Q\cdot 1$. 
We get around this as follows.  
We set $A'=A^{\frac{1}{2\pi i}}$; this power series lies in $\Zc
\langle\!\langle a,b\rangle\!\rangle$ by the definition of $A$ (cf.~\eqref{AandB}).
We then set 
$$A'(\tau)=g(\tau)\cdot A'\in \bigl(\Ec^{\rm geom}\otimes \Zc[2\pi i]\bigr)\langle\!\langle a,b\rangle\!\rangle.$$

\begin{definition} Let $\Ac'\subset \Ec^{\rm geom}\otimes 
\Zc$ denote the $\Q$-algebra generated by the coefficients of 
$A'(\tau)$. Let $\overline{\Ac}$ denote the image of $\Ac'$ in the
reduced ring $\Ec^{\rm geom}\otimes \overline{\Zc}$. Let $\overline{A}'$
be the power series obtained from $A'$ by reducing the coefficients from
$\Zc$ to $\overline{\Zc}$, and $\overline{A}'(\tau)$ to be the power
series obtained from $A'(\tau)$ by reducing the coefficients from
$\Ec^{\rm geom}\otimes \Zc$ to $\Ec^{\rm geom}\otimes \overline{\Zc}$.
We have $\overline{A}'(\tau)=g(\tau)\cdot\overline{A}'$.
\end{definition}

The coefficients of the three reduced power series $\overline{E}(\tau)$, $\overline{B}(\tau)$ and 
$\overline{A}'(\tau)$, which generate the three subalgebras $\overline{\Ec}$, 
$\overline{\Bc}$ and $\overline{\Ac}$ of $\Ec^{\rm geom}\otimes \overline{\Zc}$,
 are called $\overline{E}$-EMZs, $\overline{B}$-EMZs and $\overline{A}$-EMZs. 

Our first goal is to compare the four algebras
$\overline{\Ec}$, $\overline{\Ac}$, $\overline{\Bc}$ and $\Ec^{\rm geom}
\otimes_\Q \overline{\Zc}$.  The element $2\pi i\tau\in 
\mathcal{O}(\mathfrak{H})$ plays a special role in this comparison. It
lies in $\Ec^{\rm geom}$ since it is the coefficient of $\eps_0$ in $g(\tau)$ 
(see (\ref{two}) below), so $2\pi i\tau\otimes 1$ lies in the tensor product 
$\Ec^{\rm geom}\otimes_\Q \overline{\Zc}$, but it does not lie in $\Ec$
or $\Ac$, although it does lie in $\Bc$.  Working mod $2\pi i$ allows us
to make a much more precise statement than the simple inclusion, which
also has the advantage of showing that the three reduced algebras are
highly non-trivial.

\begin{intthm}[Thm. \ref{tensprod}] We have the equalities\footnote{Here and in the rest of the article we usually write $2\pi i\tau$ for what is 
strictly speaking the element $2\pi i\tau\otimes 1\in \Ec^{\rm geom}\otimes_\Q \overline{\Zc}$.}
$$\overline{\Ec}[2\pi i\tau]= \overline{\Ac}[2\pi i\tau]
=\overline{\Bc}=\Ec^{\rm geom}
\otimes_\Q \overline{\Zc}.$$
\end{intthm}

As noted above, the elliptic generating series $E(\tau)$ has the form 
$g(\tau)\cdot E$; thus its reduction $\overline{E}(\tau)$
has the form $g(\tau)\cdot \overline{E}$ where the reduction 
$\overline{E}$ of $E$ has coefficients in $\overline{\Zc}$.  Throughout this
article, we consider the power series $\overline{\Phi}_{\rm KZ}$ obtained by
reducing the coefficients of $\Phi_{\rm KZ}$ from $\Zc$ to $\overline{\Zc}$.
This reduced power series, a priori an associator belonging to the 
torsor of associators $M(\overline{\Zc})$, can be considered as lying in
the group $GRT_1(\overline{\Zc})$ (defined in \cite{Drinfeld}, \S 5), since the associator relations become
equivalent to the group relations for any associator whose degree 2 part is 
zero (cf. Proposition 5.9 of \cite{Drinfeld}).  The key point (Theorem \ref{mainthm}) of the proof of the equality 
$\overline{\Ec}[2\pi i\tau]=
\Ec\otimes\overline{\Zc}$ is the fact that 
\begin{equation}\label{mainthing}
\overline{E}=\widetilde\Gamma(\overline{\Phi}_{\rm KZ}),
\end{equation}
where $\widetilde\Gamma$ is the composition
$$GRT_1(\overline{\Zc})\buildrel\Gamma\over\rightarrow 
GRT_{ell}(\overline{\Zc}) \buildrel\pi\over
\rightarrow \overline{\Zc}\langle\!\langle a,b\rangle\!\rangle,$$
where $\Gamma$ is Enriquez' section (\cite{Enriquez:EllAss}, \S 4), and the projection $\pi$ maps
an element $(\lambda,f,g_+,g_-)\in GRT_{ell}(\overline{\Zc})$ 
to the component $g_+\in \overline{\Zc}\langle\!\langle a,b\rangle\!\rangle$. 
In particular, we show that the 
coefficients of $\overline{E}$ generate all of $\overline{\Zc}$
(Corollary \ref{coro35}).
The ring $\overline{\Ec}$ is the $\Q$-algebra generated by the coefficients 
of $\overline{E}(\tau)$, which, as for $E(\tau)$, are all algebraic
expressions in the coefficients of $g(\tau)$ and those of $\overline{E}$.
The delicate part of the proof consists in showing that the coefficients of 
$\overline{E}(\tau)$ can be 
``untangled'' to separately recover a set of generators of $\overline{\Zc}$ and,
with the addition of $2\pi i\tau$, a set of generators of $\Ec^{\rm geom}$.
The same arguments hold for $\overline{\Ac}$ and $\overline{\Bc}$,
except that instead of \eqref{mainthing}, we use \cite{Matthes:Thesis}, Theorem 5.4.2, to show that the coefficients
of the arithmetic parts $\overline{A}'$ and $\overline{B}$ 
generate all of $\overline{\Zc}$ and the same
untangling argument.
Thanks to the
above theorem, we call the $\Q$-algebra $\Ec^{\rm geom}\otimes_\Q \overline{\Zc}$ the {\it $\Q$-algebra of elliptic multizetas (modulo $2\pi i$)}.

\vspace{.2cm}
\subsection{The elliptic double shuffle relations for $\overline{E}$-EMZs}
The theorem in the previous section shows that the $\overline{E}$-EMZs
together with $2\pi i\tau$ form a set of generators of $\Ec^{\rm geom}\otimes
\overline{\Zc}$, as do the $\overline{A}$-EMZs and the $\overline{B}$-EMZs
taken together.  These three sets of generators are quite different
from each other.  A natural question, given a set of generators of a
ring, is to try to establish the relations they satisfy, and if possible 
a complete set thereof.  The fact that the power series 
$\overline{A}(\tau)$ and $\overline{B}(\tau)$ are group-like provides some
relations; another family, the ``Fay relations'', is partially known for 
$\overline{A}$-EMZs (cf.~\cite{Matthes:Thesis}).  Enriquez gives a
complete set of {\it associator relations} satisfied by the elliptic associator,
derived from the fact that the power series $A(\tau)$ and $B(\tau)$ induce
an automorphism of the prounipotent 2-strand torus braid group.  However,
these relations mingle the $A$-EMZs and the $B$-EMZs along with genus zero
multiple zeta values, and do not provide 
separate relations for each generating set.  For the $\overline{E}$-EMZs, 
however, we can say more, both about explicit relations
satisfied by $\overline{E}(\tau)$, and about the question of whether these
relations may be a complete set.  

The third main result of this article concerns the relations satisfied by
$\overline{E}(\tau)$.  These relations, called the
{\it elliptic double shuffle relations}, are defined in \S 4.2 using Ecalle's
mould theory.  They are extremely simple, and very similar to the genus 0 double
shuffle relations\footnote{Shuffle and stuffle, also known as double m\'elange 
in Racinet's terminology \cite{Racinet:Doubles}, and symmetrality/symmetrility in 
\'Ecalle's terminology \cite{Ecalle:Flexion}.} satisfied by $\Phi_{\rm KZ}$, in 
contrast with the elliptic associator relations, which are much more complicated than the 
genus 0 associator relations satisfied by $\Phi_{\rm KZ}$.  In fact, the elliptic double 
shuffle relations are essentially transports of the usual double shuffle equations 
under a map that extends Enriquez' section $\gamma:\grt\rightarrow \grt_{ell}$.

\begin{intthm}[rough version; precise statement given in Thm. \ref{ellipticds}] The power series
$\overline{E}(\tau)$ satisfies a double family of {\bf elliptic double shuffle 
relations} closely related to the usual double shuffle relations.
\end{intthm}

The explicit determination of the transported relations relies on several difficult 
known results, in particular \'Ecalle's crucial theorem (\cite{Ecalle:Flexion} but see 
\cite{Schneps:ARI}, Theorem 4.6.1 for a complete proof),
together with the results of \cite{Schneps:Emzv}.  

\vspace{.2cm} 
It is natural to ask whether the elliptic double shuffle relations generate 
all algebraic relations between $\overline{E}$-EMZs. We show in \S 4.2 
that the elliptic double shuffle relations are a complete set in depth 2 
(Proposition~4.6), thanks to the fact that depth 2 is too small for the real 
multiple zeta values to occur.  In higher depth, however, we naturally 
encounter problems related to the unknown transcendence properties
of the real multiple zeta values, exactly as we do when conjecturing that the
usual double shuffle relations generate all algebraic relations between
multizeta values.  In Proposition~\ref{conjs}, we show that the elliptic double
shuffle relations do form a complete set of algebraic relations between 
$\overline{E}$-EMZs under the following familiar conjectures from multizeta 
theory:

\begin{enumerate}
\item[(a)] The double shuffle relations generate all algebraic relations among the 
multiple zeta values modulo $2\pi i$.  

\item[(b)] The elliptic double shuffle Lie algebra $\ds_{ell}$ \cite{Schneps:Emzv} is isomorphic to a semi-direct product $\ds_{ell} \cong \fu \rtimes \gamma(\ds)$, where $\ds$ is the usual double shuffle Lie algebra and $\gamma$ is the
extension of Enriquez' section to $\ds$ obtained using mould theory
(cf.~\cite{Schneps:Emzv}, end of \S 1).
\end{enumerate}

Conjecture (a) is a standard conjecture in multizeta theory (cf. \cite{IKZ}, Conjecture 1). It would imply strong transcendence results for multiple zeta values, and therefore seems out of reach at the moment. Conjecture (b), however, is purely algebraic, 
and may therefore be more tractable, although it still seems difficult. It would follow for example from 
Enriquez' generation conjecture (\cite{Enriquez:EllAss}, \S 10) together with 
the conjecture that $\grt_{ell} \subset\ds_{ell}$ (an elliptic version 
of Furusho's theorem \cite{Furusho:Associators}).

\vspace{.2cm}
The last question addressed in the paper, in \S 4.3, concerns a family of 
algebraic relations satisfied by the $\overline{A}$-EMZs (Theorem \ref{FS}), 
which are the 
{\it Fay relations} on $A$-EMZs studied in \cite{BMS,Matthes:Edzv},
but here considered mod $2\pi i$; we compare them to 
the closely related {\it push-neutrality} relations. These families
are identical in depth 2 (although not in higher depths). 
In \cite{BMS,Matthes:Edzv}, the depth\footnote{The depth is
called ``length'' in \cite{BMS,Matthes:Edzv}.} 2 Fay relations are given 
explicitly
and it is shown that the Fay and shuffle relations give a complete set of 
$\Q$-linear relations between $A$-EMZs in depth 2.
The possible completeness
of the relations in all depths (depending on conjectures such as those
cited above), the precise comparison between the algebras $\overline{\Ec}$
and $\overline{\Ac}$, and above all the lifting of the questions considered
here to the situation not modulo $2\pi i$ are all topics for further research.

\subsection{Outline of the article}
The contents of this paper are organized as follows. In \S 2, we introduce the 
algebra $\Ec^{\rm geom}$ of \textit{geometric elliptic multizetas}, 
describe their relation to iterated integrals of Eisenstein series, and prove
the crucial linear independence of iterated Eisenstein integrals, as well as
the relation between $\Ec^{\rm geom}$ and the Lie algebra $\fu$.
In \S 3 we construct the elliptic generating series $E(\tau)$ and define the
$E$-EMZs to be its coefficients, and $\Ec$ to be the $\Q$-algebra they
generate.  Passing modulo $2\pi i$, we prove the main structural 
result $\overline{\Ec}[2\pi i\tau]\simeq \Ec^{\rm geom}\otimes \overline{\Zc}$
and its analogues for $\overline{\Ac}$ and $\overline{\Bc}$.
In \S 4, we study the elliptic double shuffle equations satisfied by
the mod $2\pi i$ elliptic generating series $\overline{E}(\tau)$ (or more 
precisely, the linearized version satisfied by its Lie version), and give evidence for the completeness of the resulting
system of algebraic relations between the $\overline{E}$-EMZs.
Finally, we study a family of relations satisfied by $\overline{A'}(\tau)$.
The necessary background 
concerning moulds is briefly summarized in \S 4.1.

\subsection{Notation and conventions:}

\subsubsection{Graded duals of vector spaces}

If $V=\prod_{n\geq 0}V_n$ is a $k$-vector space, then define its associated graded to be 
\begin{equation}
V^{\rm gr}:=\bigoplus_{n\geq 0} V_n.
\end{equation}
If in addition $V_n$ is finite-dimensional for every $n$, then we define the \textit{graded dual} of $V$ to be
\begin{equation}
V^{\vee}:=(V^{\rm gr})^{\vee}=\bigoplus_{n\geq 0}\Hom(V_n,k).
\end{equation}
This terminology is slightly non-standard but will be convenient for our purposes.

\subsubsection{Lie algebras}

Let $A$ be a $\Q$-algebra. We will denote by 
\begin{equation}
\ff_2(A)=\Lie_A[\![x_1,y_1]\!]
\end{equation}
the \textit{completed} (with respect to the descending central series) free Lie algebra over $A$ on two generators $x_1$ and $y_1$ with Lie bracket $[\cdot,\cdot]$. Equivalently, $\ff_2$ is the degree completion of the same Lie algebra with respect to which $x_1$ and $y_1$ both have degree equal to one. We also denote by
$
\U(\ff_2)_A
$
its (topological) universal enveloping algebra which is canonically isomorphic to $A\langle\!\langle x_1,y_1\rangle\!\rangle$, the $A$-algebra of formal power series in non-commuting variables $x_1$, $y_1$. Moreover, $\U(\ff_2)_A$ is a complete Hopf $A$-algebra, whose (completed) coproduct $\Delta$ is uniquely determined by $\Delta(w)=w \otimes 1+1\otimes w$, for $w \in \{x_1,y_1\}$.

\subsubsection{Group-like and Lie-like elements}
Let
\begin{equation}
F_2(A):=\exp(\ff_2(A)) \subset \U(\ff_2)_A
\end{equation}
be the subset of exponentials of Lie series. This is a group that may also be characterized as the 
subset of group-like elements of $\U(\ff_2)_A$, that is, $f \in \U(\ff_2)_A$ such that $\Delta(f)=f\otimes f$ and whose constant term is equal to $1$.
Likewise, the Lie algebra $\ff_2(A) \subset \U(\ff_2)_A$ is precisely the subset of Lie-like (or primitive) elements, that is, $f\in \U(\ff_2)_A$ such that $\Delta(f)=f\otimes 1+1\otimes f$.

Finally, if $A=\Q$, we will write $\ff_2$ instead of $\ff_2(\Q)$ and likewise $\U(\ff_2)$ and $F_2$ instead of $\U(\ff_2)_\Q$ and $F_2(\Q)$. 
\subsubsection{Derivations}
We denote by 
\begin{equation}
\Der_0(\ff_2) \subset \Der(\ff_2)
\end{equation}
the Lie subalgebra of continuous derivations $D: \ff_2 \rightarrow \ff_2$ which 
(i) annihilate the bracket $[x_1,y_1]$, that is, $D([x_1,y_1])=0$,
and (ii) are such that $D(y_1)$ contains no linear term.
We say that a derivation $D:\ff_2\rightarrow \ff_2$ is of homogeneous degree $m\ge 0$ if for every element
$f\in\ff_2$ of homogeneous degree $n$, $D(f)$ is of homogeneous degree $n+m$. We also let
\begin{equation}
\Der'_0(\ff_2)
\end{equation}
be the subspace of $\Der_0(\ff_2)$ spanned by derivations
of homogeneous degree $\ge 1$.

\vspace{.2cm}
{\bf Acknowledgements.}  The authors are grateful to the referees for many useful
comments.  A particular debt is owed to Benjamin Enriquez for a multitude of comments
that greatly improved the accuracy of the text.  This paper was written while Nils 
Matthes was a PhD student at Universit\"at Hamburg under the supervision of Ulf K\"uhn.

\section{Geometric Elliptic Multizetas} 
\label{sec:egeom}

In the first two sections, we respectively recall the definition of a certain Lie algebra $\fu$ of derivations \cite{Pollack:Thesis,Tsunogai:Derivations} and of iterated integrals of Eisenstein series \cite{Brown:MMV,Manin:Iterated}.

In \S 2.3, we introduce the algebra of geometric elliptic multizetas, and prove that it is isomorphic to the graded dual of the universal enveloping algebra of $\fu$. The crucial step is a linear independence result for iterated integrals of Eisenstein series, which we prove (in slightly greater generality than needed) in \S 2.4.

\subsection{A family of special derivations}
\label{liealgebra}
We begin by recalling the definition of a family of derivations, which was first considered in \cite{Tsunogai:Derivations}, also played an important role in
\cite{CEE},\footnote{The derivation $\eps_{2k}$ is denoted $\delta_{2k-2}$ in \cite{CEE}.} and was studied in detail
in \cite{Pollack:Thesis}. We refer to Section 1.5.3 for our conventions regarding free Lie algebras and their derivations which we will use freely in what follows.

First of all, note that since the commutator of $y_1 \in \ff_2$ is $\Q \cdot y_1$ (as $\ff_2$ is the completion of a free Lie algebra), every derivation $D \in \Der_0(\ff_2)$ is uniquely determined by its value on $x_1$. Similarly, the only non-zero derivation $D \in \Der_0(\ff_2)$ that annihilates $y_1$ is the derivation $\eps_0$ defined by $x_1 \mapsto y_1$, $y_1 \mapsto 0$. 
\begin{dfn}\label{epses}
For $k \geq 0$, define a derivation $\eps_{2k} \in \Der_0(\ff_2)$ by
\[
\eps_{2k}(x_1)=\ad(x_1)^{2k} (y_1)
\]
which has degree $2k$. We denote by
\[
\fu=\widehat{\Lie}(\eps_{2k}; k\ge 0)
\]
the degree completion of the Lie subalgebra generated by the $\eps_{2k}$ and let $\fu':=\Der'_0(\ff_2) \cap \fu$ be the Lie subalgebra of $\fu$ of derivations of positive degree. 
\end{dfn}
Note that
$
\Der_0(\ff_2) \cong \Der'_0(\ff_2) \rtimes \Q\eps_0
$
which restricts to an isomorphism
\begin{equation}\label{fuprod}
\fu\cong \fu'\rtimes \Q\eps_0.
\end{equation}
Also, observe that $\eps_2=-\ad([x_1,y_1])$, in particular $\eps_2$ is central in $\fu$ and that $\ad^n(\eps_0)(\eps_{2k})=0$ whenever $n\geq 2k-1$. Thus
\begin{equation}\label{uprime}
\fu'=\widehat{\Lie}(\ad^{n}(\eps_0)(\eps_{2k}); 0\leq n\leq 2k-2, \, k \geq 1).
\end{equation}
As seen above, every $\eps_{2k}$ is uniquely determined by its value on $x_1$, while $\eps_0$ is the only non-zero derivation $D \in \fu$ that annihilates $y_1$. From this, we get
\begin{prop} \label{prop:evalinjective}
The $\Q$-linear evaluation maps 
\begin{alignat*}{4}
&v_{x_1}: \Der_0(\ff_2) &&\rightarrow \ff_2, \quad && D &&\mapsto D(x_1),\\
&v_{y_1}: \Der'_0(\ff_2)&&\rightarrow \ff_2, \quad && D &&\mapsto D(y_1),
\end{alignat*}
are injective.
\end{prop}
For the applications to elliptic multizetas, it will be more natural to scale the derivations $\eps_{2k}$ as follows:
\begin{equation*}
\weps_{2k}:=\begin{cases}\frac{2}{(2k-2)!}\eps_{2k} & k>0\\ -\eps_0 & k=0.\end{cases}
\end{equation*}
In this way, $\weps_{2k}$ is the image of the Eisenstein generator $\textbf{e}_{2k}$ under the monodromy representation $\fu^{\rm geom} \rightarrow \Der_0(\ff_2)$  (cf. \cite{HM}, Theorem 22.3). 
\subsection{Iterated Eisenstein Integrals}
\label{seciei}

In a sense to be made precise below, the derivation $\eps_{2k}$ naturally corresponds to integrals of Hecke-normalized Eisenstein series of weight $2k$ (for $\SL_2(\Z)$), whereas commutators of $\eps_{2k}$ correspond to \textit{iterated integrals of Eisenstein series}. These are special cases of \textit{iterated Shimura integrals} (or \textit{iterated Eichler integrals}) of modular forms introduced by Manin \cite{Manin:Iterated}, and later generalized by Brown \cite{Brown:MMV}.\footnote{To be precise, Manin defined iterated Shimura integrals of cusp forms between base points on the upper half-plane (possibly cusps), and the extension to Eisenstein series (which requires a regularization procedure) is due to Brown.}

For $k \geq 0$, let $G_{2k}(q)$ be the Hecke-normalized Eisenstein series, defined by $G_0(q):=-1$ and for $k \geq 1$
\begin{equation*} \label{eisensteinseries}
G_{2k}(q)=-\frac{B_{2k}}{4k}+\sum_{n\ge 1} \sigma_{2k-1}(n) q^n, \quad q=e^{2\pi i\tau}
\end{equation*}
Here, $\sigma_\ell(n)=\sum_{d\vert n} d^\ell$ denotes the $\ell$-th divisor function, and the $B_{2k}$ are the Bernoulli numbers defined by
\[\frac{z}{e^z-1}=1-\frac{z}{2} + \sum_{n \geq 1}B_{2n} \frac{z^{2n}}{(2n)!}.\]
Via the exponential map $\exp: \fH \rightarrow D^*$, 
$\tau\mapsto q=\exp(2\pi i\tau)$,
from the upper half-plane to the punctured unit disk
\[
D^*=\{q\in\C\,\bigl|\, 0<\vert q \vert <1\},
\]
we may consider $G_{2k}$ as a function of either variable $q$ or $\tau$, and we shall do so according to context.

Next, we define iterated integrals of Eisenstein series. More generally, if $f(q)=\sum_{n=0}^{\infty}a_nq^n$ is such that $a_0=0$, (e.g. if $f$ is a cusp form), then the definition of the indefinite integral $\int_{\tau}^{i\infty}f(\tau_1)\dd\tau_1$ poses no problem, as by definition $f$ vanishes at $i\infty$. This is not the case for the Eisenstein series $G_{2k}$, and consequently $\int_{\tau}^{i\infty}G_{2k}(\tau_1)\dd\tau_1$ diverges. It can be regularized by setting, for $k \geq 1$,
\begin{equation*}
\int_{\tau}^{i\infty}G_{2k}(\tau_1)\dd\tau_1:=\int_{\tau}^{i\infty}\Big[G_{2k}(\tau_1)-G^{\infty}_{2k}\Big]\dd\tau_1-\int_0^{\tau}G^{\infty}_{2k}\dd\tau_1,
\end{equation*}
where $G^{\infty}_{2k}=-\frac{B_{2k}}{4k}$ is the constant term in the Fourier expansion of $G_{2k}$ (if $k=0$, a similar method works).
Note that the integral of $G_{2k}$ so defined satisfies the differential equation $\dd f(\tau)=-G_{2k}(\tau)\dd\tau$.
The definition of regularized iterated integrals of Eisenstein series in \cite{Brown:MMV}, which is a special case of Deligne's tangential base point regularization (\cite{Deligne:P1}, \S 15) generalizes this construction, and runs as follows.

Let $W=\C[\![q]\!]^{<1}$ be the $\C$-algebra of formal power series, which converge on $D=\{q \in \C \, \vert \, |q| <1 \}$. We may decompose $W=W^0 \oplus W^{\infty}$ with $W^0=q\C[\![q]\!]$ and $W^{\infty}=\C$. For a power series $f \in W$, define $f^0$ to be its image in $W^0$ under the natural projection, and define $f^{\infty} \in W^{\infty}$ likewise. For example, in the case of the Eisenstein series $G_{2k}(q)$ with $k>0$, we have
\[
G_{2k}^{\infty}=-\frac{B_{2k}}{4k}, \quad G_{2k}^0(q)=\sum_{n\ge 1} \sigma_{2k-1}(n) q^n.
\]
We denote by $T^c(W)$ the \textit{shuffle algebra} on the $\C$-vector space $W$.\footnote{An alternative name is \textit{tensor coalgebra} which explains the notation $T^c(W)$.} As a $\C$-vector space, $T^c(W)$ is simply $\bigoplus_{n \geq 0}W^{\otimes n}$ but in order to avoid confusion with the tensor algebra it is customary to write down elements of $T^c(W)$ using bar notation $[f_1|,\ldots,|f_n]$. The product of $T^c(W)$ is the shuffle product $\shuffle$, defined by
\begin{equation*}
[f_1|\ldots|f_r] \shuffle [f_{r+1}|\ldots|f_{r+s}]=\sum_{\sigma \in \Sigma_{r,s}}f_{\sigma^{-1}(1)}\ldots f_{\sigma^{-1}(r+s)},
\end{equation*}
where $\Sigma_{r,s}$ denotes the set of permutations $\sigma$ on $\{1,\ldots,r+s\}$, such that $\sigma$ is strictly increasing on both $\{1,\ldots r\}$ and on $\{r+1,\ldots,r+s\}$.

Now define a map $R: T^c(W) \rightarrow T^c(W)$ by the formula
\begin{equation*}
R[f_1|\ldots |f_n]=\sum_{i=0}^n(-1)^{n-i}[f_1|\ldots |f_i] \shuffle [f^{\infty}_n|\ldots | f^{\infty}_{i+1}].
\end{equation*}
Following \cite{Brown:MMV}, eq. (4.11), we can now make the
\begin{dfn} \label{dfn:regiterint}
Given $f_1,\ldots,f_n \in W$ as above, their \textit{regularized iterated integral} is defined as
\begin{equation*}
I(f_1,\ldots,f_n;\tau):=(2\pi i)^n\sum_{i=0}^n\int_{\tau}^{i\infty}R[f_1|\ldots |f_i]_{\dd \tau} \int_{\tau}^0[f^{\infty}_{i+1}|\ldots |f^{\infty}_n]_{\dd \tau},
\end{equation*}
where
\begin{equation*}
\int_a^b [f_1 \vert \ldots \vert f_n]_{\dd \tau}:=\idotsint\limits_{a\leq \tau_1 \leq\ldots\leq \tau_n\leq b} f_1(\tau_1)\ldots f_n(\tau_n)\dd\tau_1 \ldots \dd \tau_n.
\end{equation*}
\end{dfn}
\begin{rmk}
The reason for the $(2\pi i)^n$-prefactor is to preserve the rationality of the Fourier coefficients. More precisely, if $f_1,\ldots,f_n$ have rational coefficients (that is, $f_i \in W_{\Q}:=\Q[\![q]\!]^{<1}$), then $I(f_1,\ldots,f_n;\tau) \in W_{\Q}[\log(q)]$, where $\log(q):=2\pi i\tau$.
\end{rmk}

As is the case for usual iterated integrals (\cite{Hain:Geometry}, Sect. 2), regularized iterated integrals satisfy the differential equation
\begin{equation} \label{eqn:diffeq}
\frac{\partial}{\partial \tau}\Bigr|_{\tau=\tau_0}I(f_1,\ldots,f_n;\tau)=-f_1(\tau_0)I(f_2,\ldots,f_n;\tau_0),
\end{equation}
as well as the shuffle product formula
\begin{equation} \label{eqn:shuffle}
I(f_1,\ldots,f_r;\tau)I(f_{r+1},\ldots,f_{r+s};\tau)=\sum_{\sigma \in \Sigma_{r,s}}
I(f_{\sigma^{-1}(1)},\ldots,f_{\sigma^{-1}(r+s)};\tau).
\end{equation}
The only case of interest for us will be when $f_1,\ldots,f_n$ are given by Eisenstein series $G_{2k_1},\ldots,G_{2k_n}$. In this case, we set
\begin{equation}\label{one}
\Gc_{\uk}(\tau):=I(G_{k_1},\ldots,G_{k_n};\tau),
\end{equation}
where $\uk=(k_1,\ldots,k_n) \in (2\Z_{\geq 0})^n$ and likewise denote by 
\begin{equation*}
\cI^{\rm Eis}:=\operatorname{Span}_{\Q}\{ \Gc_{\uk}(\tau) \} \subset \cO(\fH)
\end{equation*}
the $\Q$-span of all iterated Eisenstein integrals $\Gc_{\uk}(\tau)$ for all multi-indices 
$\uk$ (including $\Gc_{\emptyset}:=1$ for the empty index). Note that $\cI^{\rm Eis}$ is a $\Q$-subalgebra of $\cO(\fH)$ by \eqref{eqn:shuffle} and that it contains $\Q[2\pi i\tau]$ as a subalgebra, since
\begin{equation}\label{two}
\Gc_0(\tau)=2\pi i\tau.
\end{equation}

\subsection{The $\tau$-evolution equation and the algebra of geometric elliptic multizetas.}
\label{goftau}

We now put together the special derivations $\weps_{2k}$ and the iterated Eisenstein integrals into a single, formal series
\begin{equation} \label{eqn:gtau}
g(\tau):={\rm id}+\sum_{\uk}\Gc_{\uk}(\tau)\weps_{\uk},
\end{equation}
where the sum is over all even multi-indices $\uk \in 2\Z_{\geq 0}^n$ for $n>0$,
and for each tuple $\uk=(k_1,\ldots,k_n)$, we set 
$\weps_{\uk}:=\weps_{k_1}\circ\ldots\circ\weps_{k_n} \in \U(\fu)$, the (completed) universal enveloping algebra of $\fu$. From \eqref{eqn:diffeq}, it is clear that $g(\tau)$ satisfies the differential equation
\begin{equation*}
\frac{1}{2\pi i}\frac{\partial}{\partial \tau}g(\tau)=-\Big(\sum_{k \geq 0}G_{2k}(\tau)\weps_{2k}\Big)g(\tau),
\end{equation*} 
and it follows that $g(\tau)$ is group-like, that is, it is the exponential $g(\tau)=\exp(r(\tau))$ of a Lie series\footnote{Here and in the following, all tensor products involving complete Lie algebras such as $\fu$ will mean completed tensor products.}
\begin{equation}\label{rtau}
r(\tau) \in \fu \otimes_{\Q} \mathcal{I}^{\rm Eis}.
\end{equation}

\begin{dfn}
Define the $\Q$-algebra $\Ec^{\rm geom}$ of geometric elliptic multizetas to be the $\Q$-algebra generated by the coefficients of $r(\tau)\cdot x_1$.
\end{dfn}
Equivalently, $\Ec^{\rm geom}$ is equal to the $\Q$-vector space linearly spanned by the coefficients 
of the series $g(\tau)\cdot e^{x_1}$, because the coefficients of each of the power series $r(\tau)\cdot x_1$ 
and $g(\tau)\cdot e^{x_1}$ can be written as algebraic expressions in the
coefficients of the other. Also, note that since every derivation in $\fu$ is uniquely determined by its value on $x_1$, the $\Q$-algebra $\Ec^{\rm geom}$ is also the same as the $\Q$-algebra spanned by the coefficients of $g(\tau)$ when
written in a basis of the $\Q$-algebra generated by the $\eps_{2k}$.  Thus
in particular we have inclusions of commutative algebras
$$\Ec^{\rm geom}\subset \Q[\Gc_{\underline{k}}(\tau)\,\bigl|\,\underline{k}\in
{\N}^n, n\ge 0]\subset \mathcal{O}(\mathfrak{H}).$$

We can now state the main result of \S 2.
\begin{thm} \label{duality}
For every $\Q$-subalgebra $A \subset \C$, there is an isomorphism
\begin{equation*}
\U(\fu)^{\vee} \otimes_{\Q} A \cong \Ec^{\rm geom} \otimes_{\Q} A
\end{equation*}
of $A$-algebras. In particular, $\Ec^{\rm geom}$ is a commutative, graded Hopf $\Q$-algebra in a natural way.
\end{thm}
\begin{proof}
The main ingredient in the proof is that the iterated Eisenstein integrals $\Gc_{\uk}(\tau)$ are linearly independent over $\C$, as functions in $\tau$. More precisely, by Corollary \ref{cor:Eis}, proved in the next section, there is a natural isomorphism
\begin{equation*}
\cI^{\rm Eis} \otimes_{\Q} A \cong T^c(V_{\mathrm{Eis}})\otimes_{\Q} A,
\end{equation*}
where $T^c(V_{\mathrm{Eis}})$ is the shuffle algebra on the $\Q$-vector space $V_{\mathrm{Eis}}$ spanned by all Eisenstein series $G_{2k}$, $k \geq 0$.

Assuming Corollary \ref{cor:Eis} for the moment, the proof of Theorem \ref{duality} proceeds as follows. Since the tensor algebra $T(V_{\mathrm{Eis}})$ is freely generated by one element in every even degree $2k \geq 0$, we get a canonical surjection
$
T(V_{\mathrm{Eis}})\rightarrow \U(\fu^{\rm gr})
$
of $\Q$-algebras, which induces by duality an injection 
\begin{equation*}
\iota: \U(\fu)^{\vee}\cong\U(\fu^{\rm gr})^{\vee}\hookrightarrow T^c(V_{\mathrm{Eis}}) \cong \cI^{\rm Eis}.
\end{equation*}
On the other hand, choosing a (homogeneous) linear basis $\cB$ of $\U(\fu^{\rm gr})$, the element $g(\tau)$ naturally defines a map
\begin{align*}
\widetilde{\iota}: \U(\fu)^{\vee} &\hookrightarrow \cI^{\rm Eis}\\
b^{\vee} &\mapsto b^{\vee}(g(\tau)),
\end{align*}
where $b^{\vee} \in \cB^{\vee}$ are the dual basis elements.
Clearly, the image of $\widetilde{\iota}$ does not depend on the choice of basis, and equals $\Ec^{\rm geom}$ by definition. On the other hand, it is easy to see that the maps $\iota,\widetilde{\iota}: \U(\fu)^{\vee} \rightarrow \cI^{\rm Eis}$ are equal, whence the result for $A=\Q$, and the general case follows simply by extension of scalars. Finally, it is well known that the universal enveloping algebra of any graded Lie algebra has a natural structure of a (cocommutative) graded Hopf algebra, thus $\U(\fu)^{\vee}$ is naturally a (commutative) graded Hopf algebra.
\end{proof}

\subsection{Linear independence}
\label{linearindependence}

In this subsection, we complete the proof of Theorem \ref{duality} by showing that the family of iterated Eisenstein integrals is linearly independent over $\C$, and that as a consequence $\cI^{\rm Eis} \otimes_{\Q} \C \cong T^c(V_{\mathrm{Eis}}) \otimes_{\Q}\C$ as $\C$-algebras. Although these results can meanwhile also be deduced from \cite{Matthes:IterIntQMod}, which proves linear independence of iterated integrals of quasimodular forms for $\SL_2(\Z)$ over the (fraction field of the) ring of quasimodular forms for $\SL_2(\Z)$ (see also \cite{Brown:II} for a similar result), we give a slightly different proof here in the special case of Eisenstein series which has the advantage that it works over a larger field of coefficients.

The main idea is to use the following general linear independence result.
\begin{thm}[{\cite[Theorem 2.1]{DDMS}}] \label{thm:auxthm}
Let $(\cA,\dd)$ be a differential algebra over a field $k$ of characteristic zero, whose ring of constants $\ker(\dd)$ is precisely equal to $k$. Let $\cC$ be a differential subfield of $\cA$ (that is, a subfield such that $\dd \cC \subset \cC$), $X$ any set with associated free monoid $X^*$. Suppose that $S \in \cA \langle \!\langle X \rangle \!\rangle$ is a solution to the differential equation
\begin{equation*}
\dd S=M \cdot S,
\end{equation*}
where $M=\sum_{x \in X}u_xx \in \cC \langle \!\langle X \rangle \!\rangle$ is a homogeneous series of degree $1$, with initial condition $S_1=1$, where $S_1$ denotes the coefficient of the empty word in the series $S$. The following are equivalent:
\begin{enumerate}
\item[(i)] The family of coefficients $(S_w)_{w \in X^*}$ of $S$ is linearly independent over $\cC$.
\item[(ii)] The family $\{u_x\}_{x \in X}$ is linearly independent over $k$, and we have
\begin{equation} \label{eqn:trivint}
\dd \cC \cap \operatorname{Span}_k(\{u_x\}_{x \in X})=\{0\}.
\end{equation}
\end{enumerate}
\end{thm}
Using this theorem, we can now prove linear independence of iterated Eisenstein integrals.
\begin{thm} \label{matthes}
The family $\{\Gc_{\uk}(\tau)\}$ is linearly independent over ${\rm Frac}(\Z[\![q]\!])$.
\end{thm}

\begin{proof}
We will apply Theorem \ref{thm:auxthm} with the following parameters:
\begin{itemize}
\item $k=\Q$, $\cA=\Q[\log(q)](\!(q)\!)$ with differential $\dd=q\frac{\partial}{\partial q}$, and $\cC={\rm Frac}(\Z[\![q]\!])$ (the latter is a differential field by the quotient rule for derivatives)
\item $X=\{a_{2k}\}_{k \geq 0}$, $u_{a_{2k}}=-G_{2k}(q)$, hence 
\[M(q)=-\sum_{k \geq 0}G_{2k}(q)a_{2k}.\]
\end{itemize}
With these conventions, it follows from \eqref{eqn:diffeq} that the formal series
\begin{equation*}
1+\int_{q}^{0}[M]_{\dd\log q}+\int_{q}^{0}[M\vert M]_{\dd\log q}+\ldots \in \cO(\fH)\langle \!\langle X \rangle \!\rangle,
\end{equation*}
with the iterated integrals regularized as in Section \ref{seciei}, is a solution of the differential equation $\dd S=M \cdot S$, with $S_1=1$. Consequently, the coefficient of the word $w=a_{2k_1}\ldots a_{2k_n}$ in $S$ is equal to $\Gc(2k_1,\ldots,2k_n;\tau)$. Moreover, since the $\Q$-linear independence of the Eisenstein series is well known (cf. e.g. \cite{Serre:Cours}, VII.3.2), it remains to verify \eqref{eqn:trivint} in our situation.
	
To this end, assume that there exist $\alpha_{2k} \in \Q$, all but finitely many of which are equal to zero, such that
\begin{equation} \label{eqn:eq1}
\sum_{k \geq 0}\alpha_{2k}G_{2k}(q) \in \dd \cC.
\end{equation}
Clearing denominators, we may assume that $\alpha_{2k} \in \Z$. Furthermore, from the definition of $\dd=q\frac{\partial}{\partial q}$, one sees that the image $\dd \mathcal C$ of the differential operator $\dd$ does not contain any constant except for zero. Therefore, the coefficient of the trivial word $1$ in \eqref{eqn:eq1} vanishes; in other words
\begin{equation*}
\sum_{k \geq 0}\alpha_{2k}G_{2k}(q)=\sum_{k \geq 1}\alpha_{2k}G^0_{2k}(q) \in q\Q[\![q]\!].
\end{equation*}
Now the differential $\dd$ is invertible on $q\Q[\![q]\!]$, and inverting $\dd$ is the same as integrating. Hence \eqref{eqn:eq1} is equivalent to
\begin{equation} \label{eqn:eq2}
\sum_{k \geq 1}\alpha_{2k}\Gc^0_{2k}(\tau) \in \cC, \quad \Gc^0_{2k}(\tau):=\int_{q}^{0}G^0_{2k}(q_1)\frac{\dd q_1}{q_1}.
\end{equation}
However, this is absurd, unless all the $\alpha_{2k}$ vanish, as we shall see now. Indeed, if $f \in \cC={\rm Frac}(\Z[\![q]\!])$, then there exists $m \in \Z \setminus \{0\}$ such that $f \in \Z[m^{-1}](\!(q)\!)$. This follows from the well known inversion formula for power series. On the other hand, the coefficient of $q^p$ in $\Gc^0_{2k}(\tau)$, for $p$ a prime number, is given by
\begin{equation*}
\frac{\sigma_{2k-1}(p)}{p}=\frac{p^{2k-1}+1}{p} \equiv \frac{1}{p} \mod \Z.
\end{equation*}
Thus, we must have 
$
\frac 1p\sum_{k \geq 1}\alpha_{2k} \in \Z[m^{-1}],
$
for every prime number $p$, in particular $\sum_{k \geq 1}\alpha_{2k}$ is divisible by infinitely many primes (namely, at least all the primes which do not divide $m$), which implies $\sum_{k \geq 1}\alpha_{2k}=0$.
	
Now assume that $k_1$ is the smallest positive, even integer with the property that $\alpha_{k_1} \neq 0$. Consider the coefficient of $q^{p^{k_1}}$ in $\Gc^0_{2k}(\tau)$, which is equal to
\begin{equation*}
\frac{\sigma_{2k-1}(p^{k_1})}{p^{k_1}}=\frac{1}{p^{k_1}}\sum_{j=0}^{k_1}p^{j(2k-1)} \equiv \begin{cases}\frac{1}{p^{k_1}} \mod \Z& \mbox{if $2k>k_1$}\\ \frac{1}{p^{k_1}}+\frac{1}{p} \mod \Z&\mbox{if $2k=k_1$.}\end{cases}
\end{equation*}
By \eqref{eqn:eq2}, we have $\frac{\alpha_{k_1}}{p}+\frac{1}{p^{k_1}}\sum_{k \geq 1}\alpha_{2k} \in \Z[m^{-1}]$, and by what we have seen before, $\sum_{k \geq 1}\alpha_{2k}=0$. Hence
$
\frac {\alpha_{k_1}}{p} \in \Z[m^{-1}],
$
for every prime number $p$, which again implies $\alpha_{k_1}=0$, in contradiction to our assumption $\alpha_{k_1} \neq 0$. Therefore, in \eqref{eqn:eq2}, we must have $\alpha_{2k}=0$ for all $k \geq 1$ and \eqref{eqn:trivint} is verified.
\end{proof}
\begin{coro} \label{cor:Eis}
The iterated Eisenstein integrals $\Gc_\uk(\tau)$ are 
$\C$-linearly independent, and for every $\Q$-subalgebra $A \subset \C$, we have a natural isomorphism of $A$-algebras
\begin{align*}
\psi_A: T^c(V_{\mathrm{Eis}}) \otimes_{\Q} A &\rightarrow \cI^{\rm Eis} \otimes_{\Q} A\\
[G_{2k_1}|\ldots|G_{2k_n}] &\mapsto \Gc_{2k_1,\ldots,2k_n}(\tau),
\end{align*}
where $V_{\mathrm{Eis}}=\operatorname{Span}_{\Q}\{G_{2k}(\tau) \, \vert \, k \geq 0\} \subset \cO(\fH)$.
\end{coro}
\begin{proof}
Since $\Q \subset \Frac(\Z[\![q]\!])$, Theorem \ref{matthes} shows in particular that the $\Gc_{\uk}$ are linearly independent over $\Q$. Since the Eisenstein series $G_{2k}$ have coefficients in $\Q$, it follows from the definition that $\Gc_{\uk} \in \Q(\!(q)\!)[\log(q)]$, and elements of $W_{\Q}[\log(q)]=\Q(\!(q)\!)[\log(q)]$ are linearly independent over $\Q$, if and only they are so over $\C$.

For the second statement, it is clear that $\psi_A$ is a homomorphism of $\Q$-algebras (since both sides are endowed with the shuffle product) and that it is surjective. The injectivity of $\psi_A$ is just the $A$-linear independence of iterated Eisenstein integrals.
\end{proof}

\begin{coro}\label{indepcor} For any $\Q$-subalgebra $A\subset\C$
(viewed inside $\mathcal{O}(\mathfrak{H})$ as constant functions),
the $\Q$-subalgebra of $\cO(\fH)$ generated by $\cI^{\rm Eis}$ and $A$ is 
canonically isomorphic to $\mathcal{I}^{\rm Eis} \otimes_{\Q} A$, and the 
$\Q$-subalgebra generated by $\Ec^{\rm geom}$ and $A$ is canonically isomorphic
to $\Ec^{\rm geom}\otimes_\Q A$.
\end{coro}

\begin{proof} 
By the previous corollary, the elements of $\cI^{\rm Eis}$ are linearly
independent over $\C$, so over $A$.  Since both $A$ and $\cI^{\rm Eis}$
are $\Q$-algebras, any element of the algebra generated by $A$ and
$\cI^{\rm Eis}$ can be expressed as a linear combination of 
elements of $\cI^{\rm Eis}$ with coefficients in $A$ which is unique
up to scalar multiplication by rationals.  This implies
the isomorphism with the tensor product.  The result carries over to
$\Ec^{\rm geom}$ trivially since $\Ec^{\rm geom}$ lies inside $\cI^{\rm Eis}$.
\end{proof}

\section{The generating series of elliptic multizetas} 
\label{sec:eds}

In the first paragraph of this section, we recall an important fact about
the elliptic associator defined by Enriquez in \cite{Enriquez:EllAss},
or more precisely, the structure of the group-like power series 
$A(\tau)$, $B(\tau)$: namely, there exist power series $A$ and $B$
(whose definitions are recalled in \eqref{AandB} below) such that
$$A(\tau)=g(\tau)\cdot A,\ \ \ \ \ B(\tau)=g(\tau)\cdot B,$$ 
where $g(\tau)$ denotes the automorphism of \eqref{eqn:gtau}. In analogy with 
this structure, we will define a new power series $E(\tau)$, which will also 
take the form $g(\tau)\cdot E$ for a power series $E\in F_2(\Zc)$. 
We call $E(\tau)$ the \textit{elliptic generating series}, and its coefficients 
the $E$-\textit{elliptic multizetas} or $E$-EMZs; similarly we call the
coefficients of $A(\tau)$ the $A$-EMZs and those of $B(\tau)$ the
$B$-EMZs.  We define $\Ec$ (resp.~$\Ac$, resp.~$\Bc$) to be the $\Q$-algebra 
generated by the $E$-EMZs (resp.~the $A$-EMZs, resp.~the $B$-EMZs;
these coefficients are Enriquez's
``elliptic analogues of multiple zeta values''). The algebras $\Ec$, $\Ac$ 
and $\Bc$ are closely related but not equal. More importantly, each
family of EMZs satisfies different algebraic relations.

In the remainder of the section, for technical reasons related to the use 
of mould theory, we work modulo $2\pi i$ in the sense explained in \S 1.2.
In particular, we consider the power series $\overline\Phi_{\rm KZ}$ 
and $\overline{E}$, which are the power series obtained from 
$\Phi_{\rm KZ}$ and $E$ by reducing the coefficients from
$\Zc$ to $\overline\Zc=\Zc/\langle (2\pi i)^2\rangle$. 
In \S 3.2, we give an expression for $\overline E$ which relates it explicitly 
to the reduced Drinfel'd associator $\overline\Phi_{\rm KZ}$, and which will be
essential in determining the algebraic relations satisfied by the 
reduced $E$-EMZs (denoted $\overline{E}$-EMZs) in \S 4.

In \S 3.3, we define the reduced power series $\overline{A}(\tau)$
and $\overline{B}(\tau)$ and determine the structure of the $\Q$-algebras $\overline{\Ec}$, $\overline{\Ac}$ and $\overline{\Bc}$ generated by the coefficients of 
$\overline{E}(\tau)$, $\overline{A}(\tau)$ and $\overline{B}(\tau)$
respectively. In particular, we will see that the three algebras are closely related to each other.

\subsection{Definition of the elliptic generating series $E(\tau)$}

Throughout this section, we use the following change of variables: $a=y_1$
and $b=x_1$. This change of variables will be applied to all the expressions
in $x_1,y_1$ encountered in the previous section, such as $g(\tau)\cdot y_1$,
and we will also express other quantities studied by B.~Enriquez in terms
of $a$ and $b$, in particular the elliptic associator.  The purpose of
this change of variables is for the application of mould theory in
\S 4.

Let $Ass_\mu$ denote the set of genus zero associators $\Phi\in F_2(\C)$
such that the coefficient of $ab$ in $\Phi$ is equal to $\mu^2/24$\ \ \cite{Drinfeld}. We will use the 
same elements $t_{01},t_{02},t_{12}$ as in \cite{Enriquez:EllAss}, \S 5.3, but 
rewritten in the variables $a$ and $b$:
\begin{equation}\label{thets}
t_{01}=Ber_{b}(-a),\ \ t_{02}=Ber_{-b}(a),\ \ t_{12}=[a,b],
\end{equation}
where 
\begin{equation*}Ber_x(y)=\frac{\ad(x)}{e^{\ad(x)}-1}(y),
\end{equation*}
so that $t_{01}+t_{02}+t_{12}=0$. Recall that Enriquez showed that a section 
from $Ass_\mu$ to the set of elliptic associators is given by mapping $\Phi\in Ass_\mu$ 
to the elliptic associator $(\mu,\Phi,A,B)$ defined by 
\begin{align}\label{AandB}
A&=\Phi(t_{01},t_{12})e^{\mu\,t_{01}}\Phi(t_{01},t_{12})^{-1}\notag\\
B&=e^{\mu\,t_{12}/2}\Phi(t_{02},t_{12})e^{b}\Phi(t_{01},t_{12})^{-1}
\end{align}
($A$ and $B$ are denoted $A_+$ and $A_-$ in \cite{Enriquez:EllAss}, Proposition 4.8).  
In this article we fix the values $\mu=2\pi i$ and $\Phi=\Phi_{\rm KZ}$, the
Drinfel'd associator, whose coefficients generate $\Zc$ \cite{Furusho:Stablederivation}.  The elliptic associator $\bigl(A(\tau), B(\tau)\bigr)$ defined in 
\S 6.2 of \cite{Enriquez:EllAss} satisfies the relations
\begin{equation}\label{uffa}
A(\tau)=g(\tau)\cdot A, \quad B(\tau)=g(\tau)\cdot B
\end{equation}
(see \S 5.2 of \cite{Enriquez:Emzv}).

The complete Lie algebra $\ff_2=\Lie[\![a,b]\!]$ is topologically generated 
by $a$ and $b$, but since the operator $Ber_b$ is invertible, we have 
\begin{equation}\label{t01bgenerate}
a=-Ber_b^{-1}(t_{01})=\Bigl(\frac{e^{\ad(b)}-1}{\ad(b)}\Bigr)(-t_{01}),
\end{equation}
so that we can just as well take $t_{01}$ and $b$ as topological generators. Similarly, we can take 
$e^{t_{01}}$ and $e^b$ as topological generators of the prounipotent group  
$F_2=F_2(\Q)=\exp(\ff_2)$, which is a priori topologically
generated by $e^a$ and $e^b$.

Let us define an automorphism $\sigma$ of $F_2(\Zc[2\pi i])$ by
\begin{align*}
&\sigma(e^{t_{01}})=\Phi_{\rm KZ}(t_{01},t_{12})e^{t_{01}}\Phi_{\rm KZ}(t_{01},t_{12})^{-1}\\
&\sigma(e^b)\ =e^{\pi i t_{12}}\Phi_{\rm KZ}(t_{02},t_{12})e^{b}\Phi_{\rm KZ}(t_{01},t_{12})^{-1}.
\end{align*}
We set
\begin{equation}\label{defE}
E=\sigma(a),\ \ \ C=\exp(E)=\sigma(e^a).
\end{equation}
The automorphism $\sigma$ extends to an automorphism of the
completed enveloping algebra $\U(\ff_2)$, and restricts to an automorphism
of $\ff_2$.  Thus the power series $E=\sigma(a)$ is Lie-like.  All
Lie-like and group-like power series discussed in this section are
contained in the free non-commutative power series ring 
$R\langle\!\langle a,b\rangle\!\rangle$ where $R$ is either $\Ec^{\rm geom}\otimes
\Zc[2\pi i]$ or $\Ec^{\rm geom}\otimes \overline{\Zc}$. When
we speak of the ring generated by the coefficients of such a power series,
we mean that we take coefficients of the power series written in any
linear basis of $R\langle\!\langle a,b\rangle\!\rangle$ (for example the
basis of monomials in $a$ and $b$), all of which lie in $R$, and consider
the subring of $R$ generated by these coefficients. The ``degree'' is the
degree in the variables $a$ and $b$.

Up to degree 5, the explicit expansion of $E$ is given by
\begin{equation}\label{cof}
E=a-\frac{\pi i}{2}c_3+\frac{\pi i}{12}[c_1,c_3]
+\zeta(3)c_5+\frac{\pi^2}{36}[c_1,c_4]+\frac{\pi^2}{9}[c_3,c_2]+\cdots
\end{equation}
where $c_i=\ad(a)^{i-1}(b)$.

In analogy with \eqref{uffa}, we now set
\begin{equation}\label{taudefs}
E(\tau)=g(\tau)\cdot E\ \ {\rm and}\ \ C(\tau)=g(\tau)\cdot C=g(\tau)\cdot 
\sigma(e^a).
\end{equation}
These power series lie in $\bigl(\Ec^{\rm geom}\otimes \Zc[2\pi i]\bigr)\langle\!\langle a,b\rangle\!\rangle.$

\begin{dfn} 
The power series $E(\tau)$ is called the {\it elliptic generating 
series}, and its coefficients are called the {\it $E$-EMZs} or
{\it $E$-elliptic multizetas}.
For $\underline{k}=(k_1,\ldots,k_r)$ we write $E(\underline{k})$ for the 
coefficient in $E(\tau)$ of the monomial $c_{k_1}\cdots c_{k_r}$, which
lies in the tensor product $\Ec^{\rm geom}\otimes \Zc[2\pi i]$.
The $\Q$-algebra generated inside $\Ec^{\rm geom}\otimes \Zc[2\pi i]$ by 
the $E$-elliptic multizetas $E(\underline{k})$ is denoted by $\Ec$.
\end{dfn}

\begin{lem}\label{2piiEAB}
The power series $E(\tau)$ is Lie-like and $C(\tau)$ is group-like.  The element $2\pi i$ appears
as a coefficient in each of the three  power series $A(\tau)$, $B(\tau)$ and $E(\tau)$
expanded in the variables $c_i$.  The element $2\pi i$ does not lie in the $\Q$-algebra $\mathcal{A}'$ generated
by the coefficients of $A'(\tau)$, but $\pi^2$ does. \end{lem}

\begin{proof}Since $g(\tau)$ can be considered as an automorphism of the universal 
enveloping algebra of $\ff_2$, it preserves the Lie algebra
$\ff_2\otimes_\Q (\Ec^{\rm geom}\otimes_\Q\Zc[2\pi i])$; thus
$E(\tau)$ is Lie-like, and $C(\tau)$ is group-like.   To check that a rational multiple
of $2\pi i$ occurs as a coefficient in each of the three power series in the statement, it
suffices to give their expansions in the $c_i$ in low weights using the explicit formulas 
\eqref{AandB}, \eqref{uffa}, \eqref{cof} and \eqref{taudefs}, together with formula \eqref{eqn:gtau}
defining $g(\tau)$.  For the first three, we obtain
\begin{align*}
E(\tau)&=a+\bigl(G_2(\tau)-{\frac{\pi i}{2}}\bigr)c_3+{\frac{\pi i}{12}}G_0(\tau)c_4+{\frac{\pi i}{12}}[c_1,c_3]+\cdots\\
A(\tau)&=1-2\pi i c_1-2\pi^2c_1^2+\pi ic_2+\cdots\\
B(\tau)&=1+c_1+{\frac{1}{2}}c_1^2+\pi ic_2+{\frac{\pi^2}{3}}c_3+\cdots
\end{align*}
which shows that the coefficient $2\pi i$ appears as a coefficient in low weight.

For the final statement, it is easy to see that $2\pi i$ does not appear in the ring of coefficients of $A'(\tau)$
since this power series is given by applying $g(\tau)$ to the product of 
three terms
$$A'=\Phi_{\rm KZ}(t_{01},t_{12})\,e^{t_{01}}\,\Phi_{\rm KZ}(t_{01},t_{12})^{-1},$$
none of which have $2\pi i$ in their coefficient rings, since these lie in
${\R}$ and $2\pi i$ does not.  Since $\Phi_{\rm KZ}$ has $\zeta(2)$ as a coefficient, we may ask whether
$\pi^2$ lies in the coefficient ring of $A'(\tau)$.  The expansion of the power series $A'(\tau)$ is quite complicated and necessitates the help of a computer.  It is necessary to go up to weight 5 in order to find enough coefficients to
isolate $\pi^2$.  In weight 5, however, we find that the sum of the coefficient of $c_1^3c_2$ and
$c_2c_1^3$ in the expansion of $A'(\tau)$ is equal to ${\frac{4\pi^2-1}{24}}$, which shows that $\pi^2$ does
lie in the coefficient ring $\mathcal{A}'$.\end{proof}

\begin{lem}\label{sameoldthing}
The underlying vector space of $\Ec$ is spanned by the coefficients of  $C(\tau)$.  
\end{lem}

\begin{proof}
Let $\Ec'$ denote the $\Q$-vector space spanned by the coefficients 
of $C(\tau)$.  We first note that $\Ec'$ is in fact a $\Q$-algebra, because 
$C(\tau)$ is a group-like power series, which means that the product of two of 
its coefficients can be written as a linear combination of such by using the 
(multiplicative) shuffle relations.  Next, we note that since 
$C(\tau)=\exp\,E(\tau)$, the coefficients of $C(\tau)$ can be expressed as 
algebraic combinations of the coefficients of $E(\tau)$, 
so they lie inside the subring $\Ec\subset\Ec^{\rm geom}\otimes\Zc[2\pi i]$.
Thus $\Ec'\subset \Ec$. Conversely, since $E(\tau)=\log\,C(\tau)$, 
the coefficients of $E(\tau)$ are all algebraic and thus linear 
combinations of those of 
$C(\tau)$, so they lie in $\Ec'$, so $\Ec\subset\Ec'$. This completes the proof.
\end{proof}

\subsection{An expression for $E$ modulo $2\pi i$} \label{ssec:3.2}
From now until the end of this article, we work modulo $2\pi i$, in the
sense that if a series has coefficients in $\Zc[2\pi i]$, then we reduce these
modulo the ideal generated by $2\pi i$. The quotient ring $\overline\Zc$ 
is equal to the quotient of $\Zc$ by $(2\pi i)^2$, or equivalently, by 
$\zeta(2)$.  We use overlining to denote the reduced objects. 
The goal of the section is to obtain an expression for $\overline E$ that 
relates it directly to the reduced Drinfeld associator $\overline\Phi_{\rm KZ}$. 

In order to approach this result, we will move from the Lie algebra
of derivations over to power series in $a$ and $b$ by using the map
given by evaluation at $a$.  This is important because it allows us to 
compare derivations with power series in $a$ and $b$ such as 
$\overline\Phi_{\rm KZ}$. 

Let $v_a$ denote the linear map given by evaluation at $a$.  In 
Proposition~\ref{prop:evalinjective} we considered this map restricted to
$D\in \Der'_0(\ff_2)$; we have $v_a(D)=D(a)$ for 
$D\in \Der'_0(\ff_2)$.  
Let the push-operator be the linear automorphism of $\Q\langle\!\langle a,b
\rangle\!\rangle$ defined by cyclically permuting the powers 
of $a$ between the letters $b$ in a 
monomial: 
\begin{equation}\label{pushop}
push(a^{k_0}b\cdots ba^{k_r})=a^{k_r}ba^{k_0}b\cdots a^{k_{r-1}},
\end{equation}
extended to polynomials and power series by linearity.  A power series is 
said to be {\it push-invariant} if $push(p)=p$.  
It is shown in \cite{Schneps:Emzv}, Lemma 2.1.1, that the restriction of $v_a$ to the Lie 
subalgebra $\Der'_0(\ff_2)$, which is injective by Proposition 
\ref{prop:evalinjective}, has 
image equal to the space of push-invariant Lie series
$\ff_2^{\rm push}\subset \ff_2$.
The map $v_a$ transports the Lie bracket on $\Der'_0(\ff_2)$ to a Lie
bracket $\langle\cdot,\cdot\rangle$ on $\ff_2^{\rm push}$ as follows:
\begin{equation}\label{expa}
\langle D(a),D'(a)\rangle=[D,D'](a). 
\end{equation}
We also use $v_a$ to transport the exponential map 
$\exp:\Der'_0(\ff_2)\rightarrow {\rm Aut}(\ff_2)$ to an exponential map
$\exp_a$ which makes the following diagram commute:
\begin{equation}\label{diag0}
\xymatrix{\Der'_0(\ff_2)\ar[r]^{\exp}\ar[d]^{v_a}&{\rm Aut}'_0(\ff_2)\ar[d]_{v_a}\\
\ff_2^{\rm push}\ar[r]^{\exp_a}&\ff_2,}
\end{equation}
where ${\rm Aut}'_0(\ff_2)={\rm exp}\bigl({\rm Der}'_0(\ff_2)\bigr)$.  We
observe that the right-hand vertical map $v_a$ is injective
on ${\rm Aut}'_0(\ff_2)$. Indeed, if $\exp(D)\cdot a=\exp(D')\cdot a$ for
two derivations $D\ne D'\in \Der'_0(\ff_2)$, then by comparing terms of lowest degree on both sides, we see that $D(a)=D'(a)$, hence $D=D'$ by Proposition \ref{prop:evalinjective}.

This shows that $v_a$ is injective on ${\rm Aut}'_0(\ff_2)$, which by the
diagram \eqref{diag0} then shows that $\exp_a$ is also injective.  
Let $\mathcal{G}_a$ denote the
image of $\ff_2^{\rm push}$ under $\exp_a$, or equivalently, the
image of ${\rm Aut}'_0(\ff_2)$ under $v_a$. Then $\mathcal{G}_a$ is
a set of elements in $\ff_2$, which forms a group when equipped with the 
group law transported from ${\rm Aut}'_0(\ff_2)$ by $v_a$.  This group law, 
which we denote by $\star_a$, is compatible with the Campbell-Hausdorff law on 
the Lie algebra $\ff_2^{\rm push}$, since for two derivations 
$D,D'\in \Der'_0(\ff_2)$, we have
\begin{align}\label{stara}\exp_a\bigl(D(a)\bigr)\star_a\exp_a\bigl(D'(a)\bigr)
&=\bigl(\exp(D)\cdot a\bigr)\star_a\bigl(\exp(D')\cdot a\bigr)\notag\\
&=\bigl(\exp(D)\circ\exp(D')\bigr)\cdot a\notag\\
&=\exp\bigl(ch_{\Der'_0(\ff_2)}(D,D')\bigr)\cdot a\notag\\
&=\exp_a\bigl(ch_{\ff_2^{\rm push}}\bigl(D(a),D'(a)\bigr)\bigr).
\end{align}
We also have $\exp_a(0)=v_a\bigl(\exp(0)\bigr)=v_a({\rm id})=a$, so in fact the
element $a$ is the unit element of the group 
$\mathcal{G}_a$ equipped with the multiplication $\star_a$.
Explicitly, for $D\in \Der'_0(\ff_2)$, we have
\begin{equation}\label{aminus1}
\exp_a\bigl(D(a)\bigr)=v_a\bigl(\exp(D)\bigr)=\exp(D)\cdot a=a+D(a)+\frac{1}{2}D^2(a)+\cdots
\end{equation}

\vspace{.2cm}
Let $\grt_{ell}$ be the elliptic Grothendieck-Teichm\"uller Lie algebra
defined by Enriquez in \S 5.6 of \cite{Enriquez:EllAss}.\footnote{We work with the completed version which is denoted $\widehat{\grt}_{ell}$ in \cite{Enriquez:EllAss}.}  Not surprisingly, this 
Lie algebra
will be an essential tool in proving our results.  Let us recall some of
the basic facts concerning it.  Firstly, Enriquez showed that there is a 
natural Lie morphism $\grt_{ell}\rightarrow \Der_0(\ff_2)$. It was 
further shown in \cite{Schneps:Emzv}, Equation (1.2.4), that this map is injective.\footnote{Note that what is denoted $\Der_0(\ff_2)$ in \cite{Schneps:Emzv} is denoted here by $\Der'_0(\ff_2)$.}  We
will identify $\grt_{ell}$ with its image in $\Der_0(\ff_2)$.

Enriquez also proved the following results. There is a canonical surjection 
$\grt_{ell}\rightarrow \grt$.  Let $\mathfrak{r}_{ell}$ denote the kernel; 
then it is easy to see that $\fu\subset \mathfrak{r}_{ell}$.  Finally, 
Enriquez gave a section $\gamma:\grt\rightarrow\grt_{ell}$ 
of the canonical surjection, and showed that $\grt_{ell}$ has the form
of a semi-direct product
$$\grt_{ell}\cong \mathfrak{r}_{ell}\rtimes\gamma(\grt).$$

We write $\gamma_a$ for the composition map $v_a\circ\gamma$, so that
\begin{equation}\label{gammaplus}
\gamma_a:\grt\rightarrow \ff_2^{\rm push}.
\end{equation}

Let $\exp^\odot$ denote the (``twisted Magnus'') exponential map 
$\exp^\odot:\grt\rightarrow GRT$ (\cite{Racinet:Thesis}, (2.14), where it is
denoted $\circledast$).  Recall that $\Der^*\bigl({\rm Lie}[[x,y]]\bigr)$ is the space
of derivations that annihilate $x$ and take $y$ to a bracket $[y,f]$ and
$z=-x-y$ to a bracket $[z,g]$ for some $f,g\in {\rm Lie}[x,y]$.  Writing
${\rm Aut}^*\bigl({\rm Lie}[[x,y]]\bigr)$ for the group of automorphisms
$\exp(D)$ with $D\in \Der^*\bigl({\rm Lie}[[x,y]]\bigr)$,
we have the commutative diagram
\begin{equation}\label{diag}
\xymatrix{
{\rm Der}^*({\rm Lie}[[x,y]])\ar[d]_{\exp}&\grt\ar[l]\ar[d]_{\exp^\odot}\ar[r]^\gamma&\grt_{ell}\ar[d]^{\exp}\ar[r]^{v_a}&\ff_2^{\rm push}\ar[d]^{\exp_a}\\
{\rm Aut}^*({\rm Lie}[[x,y]])&GRT\ar[l]\ar[r]^\Gamma&GRT_{ell}\ar[r]^{v_a}&\mathcal{G}_a.}
\end{equation}
where $\Gamma$ is the group homomorphism that makes the middle square commute,
and the upper map $\grt\rightarrow \Der^*({\rm Lie}[[x,y]])$ in the left-hand 
square is the map that takes a Lie element $\psi\in\ff_2$ to the associated 
{\it Ihara derivation} $D_\psi$ defined by 
\begin{equation}\label{iharader}
D_\psi(x)=0,\ \ \ \  D_\psi(y)=[\psi(x,y),y].
\end{equation}
Ihara \cite{Ihara:Tatetwist,Ihara:Stable} studied these derivations
in detail, and in particular, he showed that if 
$\Psi=\exp^\odot(\psi)$ and
$A_\Psi$ denotes the automorphism $\exp(D_\psi)$ of $\U({\rm Lie}[[x,y]])$, then 
\begin{equation}\label{iharaaut}
A_\Psi(x)=x,\ \ \ \ A_\Psi(y)=\Psi\ y\ \Psi^{-1}.
\end{equation}

\vspace{.2cm}
We can now state the main result of this subsection.

\begin{thm} \label{mainthm}
Let $\overline{E}$ be obtained from $E$ by reducing the coefficients from
$\overline{\Zc}$ to $\Zc/\langle (2\pi i)^2\rangle$. Then
\begin{equation*}
\overline{E}=\Gamma(\overline\Phi_{\rm KZ})\cdot a.
\end{equation*}
\end{thm}

\begin{proof} Let $\psi\in\grt$, and let $\Psi=\exp^\odot(\psi)\in GRT$.
Then $\gamma(\psi)\in \grt_{ell}\subset \Der_0(\ff_2)$ and 
$\Gamma(\Psi)=\exp\bigl(\gamma(\psi)\bigr)\in GRT_{ell}\subset
{\rm Aut}_0(\ff_2)$, the group of automorphisms preserving $[a,b]$.
The proof is based on a result from \cite{Enriquez:EllAss}, Lemma-Definition
4.6, which states that the automorphism $\Gamma(\Psi)$ acts as 
follows: 
\begin{align} 
\Gamma(\Psi)(t_{01})&=\Psi(t_{01},t_{12})t_{01}\Psi(t_{01},t_{12})^{-1}
\label{GammaPsi}\\
\Gamma(\Psi)(b)\,\ \ &=\log\bigl(\Psi(t_{02},t_{12})e^b\Psi(t_{01},t_{12})^{-1}\bigr),\notag
\end{align}
where $t_{01}$ is as in \eqref{thets}.  Recall from \eqref{t01bgenerate}
that we can take $t_{01}$ and $b$ as generators of $\ff_2$. 

\vspace{.2cm}
Recall that $\overline\Phi_{\rm KZ}\in GRT(\overline\Zc)$.
(This is the reason for which we work mod $2\pi i$, since the
term $-\zeta(2)[x,y]$ in $\Phi_{\rm KZ}$ means that it does not lie in $GRT$,
preventing us from taking advantage of the results on $\grt_{ell}$.)
Let $\overline\sigma$ be 
the automorphism of $F_2(\overline\Zc)$ obtained from 
$\sigma$ by reducing modulo $2\pi i$, that is,
\begin{align}\label{defAprime}
&\overline\sigma(e^{t_{01}})=\overline\Phi_{\rm KZ}(t_{01},t_{12})e^{t_{01}}\overline\Phi_{\rm KZ}(t_{01},t_{12})^{-1}=\overline{A'}\\
&\overline\sigma(e^b)\ =\ \overline\Phi_{\rm KZ}(t_{02},t_{12})e^{b}\overline\Phi_{\rm KZ}(t_{01},t_{12})^{-1}=\overline{B},\notag
\end{align}
where we set $A'=A^{1/2\pi i}$ and $\overline{A'}$ and 
$\overline{B}$ denote the reductions of $A'$ and $B$ mod $2\pi i$.

Comparing with the values of $\Gamma(\overline\Phi_{\rm KZ})$ from
\eqref{GammaPsi} on the generators $t_{01}$,
$b$ of $\ff_2$, we find that $\overline\sigma=\Gamma(\overline\Phi_{\rm KZ})$.
Since $E=\sigma(a)$ by \eqref{defE}, we find that
$$\overline{E}=\overline\sigma(a)=\Gamma(\overline{\Phi}_{\rm KZ})\cdot a,$$
which concludes the proof.\end{proof}

\begin{coro}\label{coro35} The $\Q$-algebra generated by the coefficients of 
$\overline E$ is all of $\overline\Zc$.
\end{coro}

\begin{proof} Set $\phi_{\rm KZ}=\log^\odot(\overline\Phi_{\rm KZ})$, so that 
$\phi_{\rm KZ}\in \grt\otimes_\Q\overline\Zc$. 
We first show that the coefficients of $\phi_{\rm KZ}$ (written in a basis of
$\grt$) multiplicatively generate the same ring as the coefficients of
$\overline{\Phi}_{\rm KZ}$, namely all of $\overline{\Zc}$. To do this, we
use an argument analogous to the one in the proof of Lemma \ref{sameoldthing}.
Let $\overline{\Zc}'$ denote the $\Q$-algebra generated multiplicatively by the 
coefficients of $\phi_{\rm KZ}$.  We of course know that the 
$\Q$-vector space $\overline{\Zc}$ spanned by the (reduced) multizeta values 
which are the coefficients of $\overline{\Phi}_{\rm KZ}$ is actually a 
$\Q$-algebra, since $\overline{\Phi}_{\rm KZ}$ is group-like.  
The definition of the twisted Magnus exponential 
(cf.~\cite{Racinet:Thesis}, (2.14), or \cite{Schneps:ARI}, (3.5.2)) shows that
the coefficients of $\overline{\Phi}_{\rm KZ}$ are all algebraic expressions in
the coefficients of $\phi_{\rm KZ}$; thus $\overline{\Zc}\subset \overline{\Zc}'$.
Similarly, since $\phi_{\rm KZ}=\log^\odot\overline{\Phi}_{\rm KZ}$, the
coefficients of $\phi_{\rm KZ}$ are also all algebraic expressions in elements of 
$\overline{\Zc}$; thus $\overline{\Zc}'=\overline{\Zc}$; in other words, the
coefficients of $\phi_{\rm KZ}$ multiplicatively generate $\overline{\Zc}$.

Since $\gamma$ is an injective map from $\grt$ to $\grt_{ell}$, the coefficients
of $\gamma(\phi_{\rm KZ})$ in a basis of $\grt_{ell}$ also generate all of
$\overline{\Zc}$.  Recall that Enriquez showed that $\grt_{ell}$ is isomorphic
to a semi-direct product of two of its subspaces, $\gamma(\grt)\rtimes\mathfrak{r}_{ell}$, and that $\eps_{2k}\in \mathfrak{r}_{ell}$ for $k\ge 0$.  Therefore
we see that $\eps_0\notin \gamma(\grt)\subset \grt_{ell}\subset \Der_0(\ff_2)$.
Since the natural bigrading on $\Der_0(\ff_2)$ restricts to a bigrading
on $\gamma(\grt)$ (cf.~\cite{Enriquez:EllAss}), we find that in fact
$\gamma(\grt)\subset \Der'_0(\ff_2)$.  Therefore, by Proposition~\ref{prop:evalinjective}, the evaluation map $v_a$ is injective on 
$\gamma(\grt)$, so the coefficients of $\gamma(\phi_{\rm KZ})\cdot a$ in
a basis of $v_a(\grt_{ell})$ also generate $\overline{\Zc}$.  By the
same argument as above, thanks to the definition of $\exp_a$ in
\eqref{aminus1},
the coefficients of $\exp_a\bigl(\gamma(\phi_{\rm KZ})\cdot a\bigr)$ 
then span $\overline{\Zc}$.
However, by \eqref{aminus1} and the diagram \eqref{diag}, we have
\begin{equation}\label{gammaphi}
\exp_a\bigl(\gamma(\phi_{\rm KZ}\cdot a\bigr)=\exp\bigl(\gamma(\phi_{\rm KZ})
\bigr)\cdot a=\Gamma(\overline{\Phi}_{\rm KZ})\cdot a=\overline{E},
\end{equation}
which completes the proof.
\end{proof}

\vspace{.2cm}
\subsection{Structure of the $\Q$-algebras $\overline\Ec$,
$\overline{\Ac}$ and $\overline{\Bc}$.}

Since the rings generated by the coefficients of $E$, $A$ and $B$ all
lie inside $\Zc[2\pi i]$ and the ring generated by the coefficients of
$g(\tau)$ is $\Ec^{\rm geom}$, and since $E(\tau)=g(\tau)\cdot E$,
$A(\tau)=g(\tau)\cdot A$ and $B(\tau)=g(\tau)\cdot B$, the rings
$\Ec$, $\Ac$ and $\Bc$ generated by the coefficients of $E(\tau)$,
$A(\tau)$ and $B(\tau)$ respectively are all contained inside the
ring generated inside $\mathcal{O}(\mathfrak{H})$ by the
subrings $\Ec^{\rm geom}$ and $\Zc[2\pi i]$, which as we saw is isomorphic
to the tensor product of these two rings. We therefore may and will view $\Ec,\Ac$ and $\Bc$ as subrings of $\Ec^{\rm geom}\otimes_\Q \Zc[2\pi i]$:
\begin{equation}
\Ec,\Ac,\Bc\subset\Ec^{\rm geom}\otimes_\Q \Zc[2\pi i].
\end{equation} 
We have
\begin{equation}
\Ec^{\rm geom}\otimes_\Q \Zc[2\pi i]\rightarrow
\bigl(\Ec^{\rm geom}\otimes_\Q \Zc[2\pi i]\bigr)/\langle 1\otimes 2\pi i\rangle
\simeq \Ec^{\rm geom}\otimes_\Q \overline{\Zc}.
\end{equation}
We saw in Lemma \ref{sameoldthing} that $2\pi i\in \Ec$. The
$\Q$-algebra $\overline{\Ec}$ generated by the coefficients of 
$\overline{E}(\tau)$ is equal to the quotient of $\Ec\subset
\Ec^{\rm geom}\otimes \Zc[2\pi i]$ by the 
intersection of $\Ec$ with the ideal $\langle 1\otimes 2\pi i\rangle$, 
so we have an inclusion
$$\overline{\Ec}\subset \Ec^{\rm geom}\otimes_\Q \overline{\Zc}.$$

Recall from Definitions 1.1 and 1.2 that we set
$\overline{A'}(\tau)=g(\tau)\cdot \overline{A'}$ and
$\overline{B}(\tau)=g(\tau)\cdot \overline{B}$, where $\overline{A'}$ and
$\overline{B}$ are as in (\ref{defAprime}) above. 
We write $\overline{\Ac}$ for
the $\Q$-algebra generated by the coefficients of $\overline{A'}(\tau)$
and $\overline{\Bc}$ for that generated by the coefficients of 
$\overline{B}(\tau)$.  Then like for $\overline{\Ec}$, we have inclusions
$$\overline{\Ac},\overline{\Bc}\subset \Ec^{\rm geom}\otimes_\Q \overline{\Zc}$$
(see Definition 1.2).
The goal of this paragraph is to compare the 
subrings $\overline{\Ec}$, $\overline{\Ac}$ and $\overline{\Bc}$
of $\Ec^{\rm geom}\otimes_\Q \overline{\Zc}$.

\vspace{.2cm}
\begin{thm}\label{tensprod}
We have the following equalities:
$$\overline{\Ec}[2\pi i\tau]=\overline{\Ac}[2\pi i\tau]=\overline{\Bc}
=\Ec^{\rm geom}\otimes_\Q \overline{\Zc}.$$
\end{thm}

\begin{proof}
Let $r(\tau)=\log\,g(\tau)\in \Der_0(\ff_2)$.  Recall from \eqref{defAprime}
that $\overline{A'}=\overline{\sigma}(e^{t_{01}})$,
$\overline{B}=\overline{\sigma}(e^b)$, and
$\overline{E}=\overline{\sigma}(a)$, where 
$\overline{\sigma}=\Gamma(\overline\Phi_{\rm KZ})=\exp\bigl(\gamma(\phi_{\rm KZ})\bigr).$
Let us write $ch_{[,]}=ch_{\Der_0(\ff_2)}$ for the Campbell-Hausdorff law in 
the derivation algebra $\Der_0(\ff_2)$. We set 
$$\delta(\tau)=ch_{[\,,\,]}\bigl(r(\tau),\gamma(\phi_{\rm KZ})\bigr)\in \Der_0(\ff_2).$$
Then we have
\begin{align}\label{threeseries}
\begin{cases}
\log\,\overline{A'}(\tau)=g(\tau)\cdot \log\,\overline{A'}=\exp\bigl(r(\tau))\circ
\exp\bigl(\gamma(\phi_{\rm KZ})\bigr)\cdot t_{01}=\exp\bigl(\delta(\tau)\bigr)\cdot t_{01}\\
\log\,\overline{B}(\tau)=g(\tau)\cdot \log\,\overline{B}=\exp\bigl(r(\tau))\circ
\exp\bigl(\gamma(\phi_{\rm KZ})\bigr)\cdot b=\exp\bigl(\delta(\tau)\bigr)\cdot b\\
\overline{E}(\tau)=g(\tau)\cdot \overline{E}=\exp\bigl(r(\tau))\circ
\exp\bigl(\gamma(\phi_{\rm KZ})\bigr)\cdot a=\exp\bigl(\delta(\tau)\bigr)\cdot a.
\end{cases}
\end{align}
We also set
$$\mathfrak{a}(\tau)=\delta(\tau)\cdot t_{01},\ \ \mathfrak{b}(\tau)=\delta(\tau)\cdot b,\ \ \mathfrak{e}(\tau)=\delta(\tau)\cdot a.$$

\vspace{.3cm}
\noindent {\bf Step 1: The case of $\overline{\Bc}$.}  This case turns out
to be the easiest one of the three, for the reason that by Proposition \ref{prop:evalinjective}, the map $v_b$ evaluating derivations on $b\in \ff_2$ is injective 
on all of $\fu$, which is not the case for $v_a$ or $v_{t_{01}}$.  
Let $V_b\subset \ff_2$ denote the vector space $v_b\bigl(\Der_0(\ff_2)\bigr)$.
Recall the whole situation with the exponential map $\exp_a$ 
and the group $\mathcal{G}_a$ that we introduced in
\eqref{expa}, \eqref{diag0}, \eqref{stara}, \eqref{aminus1}. 
Thanks to the injectivity of $v_b$, we can set up the analogous situation
for $b$ instead of $a$, but now using all of $\Der_0(\ff_2)$.

We first transport the Lie bracket from $\Der_0(\ff_2)$ onto $V_b$ via $v_b$,
setting 
$$\langle D(b),D'(b)\rangle_b=[D,D'](b),$$
which makes $V_b$ into a Lie algebra.
We then define $\exp_b$ to be the map that makes the diagram
\begin{equation}\label{diagb}
\xymatrix{\Der_0(\ff_2)\ar[r]^{\exp}\ar[d]^{v_b}&{\rm Aut}_0(\ff_2)
\ar[d]_{v_b}\\
V_b\ar[r]^{\exp_b}&\ff_2}
\end{equation}
commute.

Exactly as in the case of $a$, we can show that the right-hand vertical map 
$v_b$ induced on ${\rm Aut}_0(\ff_2)$ is still injective.  
Thus $\exp_b$ is also injective on $V_b$.  We write 
$$\mathcal{G}_b=\exp_b(V_b)=v_b\bigl({\rm Aut}_0(\ff_2)\bigr)\subset \ff_2,$$
and equip this set with a group law $\star_b$ in analogy with 
\eqref{stara}, transported from ${\rm Aut}_0(\ff_2)$ by $v_b$:
\begin{align}\label{starb}
\exp_b\bigl(D(f)\bigr)\star_b\exp_b\bigl(D'(b)\bigr)&=v_b\bigl(\exp(D)\bigr)
\star_b v_b\bigl(\exp(D')\bigr)\notag\\
&=v_b\bigl(\exp(D)\circ\exp(D')\bigr)\notag\\
&=\bigl(\exp(D)\circ\exp(D')\bigr)\cdot b.
\end{align}
We write $\log_b:\mathcal{G}_b\rightarrow V_b$ for the inverse of $\exp_b$, so
that
\begin{equation}\label{logb}
\log_b\bigl(\exp(D)\cdot b\bigr)=D(b).
\end{equation}

We will use all this information in the following calculation.  
By \eqref{threeseries} and the diagram \eqref{diagb}, we have
\begin{equation}\label{B1}
\log\,\overline{B}(\tau)=\exp\bigl(\delta(\tau)\bigr)\cdot b=\exp_b\bigl(\delta(\tau)\cdot b\bigr) \in \mathcal{G}_b,
\end{equation}
so
\begin{equation}\label{B2}
\mathfrak{b}(\tau):=\delta(\tau)\cdot b=\log_b\,\log\,\overline{B}(\tau)\in V_b.
\end{equation}

\vspace{.1cm}
\noindent Thus by \eqref{logb} with $D=\delta(\tau)=ch_{[\,,\,]}\bigl(r(\tau),\gamma(\phi_{KZ})\bigr)$, we have
\begin{align*}
\mathfrak{b}(\tau)&=\delta(\tau)\cdot b\notag\\
&=ch_{[\,,\,]}\bigl(r(\tau),\gamma(\phi_{\rm KZ})\bigr)\cdot b\notag\\
&=r(\tau)\cdot b+\gamma(\phi_{\rm KZ})\cdot b+\frac{1}{2}[r(\tau),\gamma(\phi_{\rm KZ})]\cdot b+\cdots\notag\\
&=r(\tau)\cdot b+\gamma(\phi_{\rm KZ})\cdot b+\frac{1}{2}\langle r(\tau)\cdot b,
\gamma(\phi_{\rm KZ})\cdot b\rangle_b+\cdots,
\end{align*}
which we rewrite as
\begin{equation}\label{frakbtau}
\mathfrak{b}(\tau)=r(\tau)\cdot b+\gamma(\phi_{\rm KZ})\cdot b+s(\tau)\cdot b,
\end{equation}
where $s(\tau)$ is the sum of all the bracketed terms in 
$ch_{[\,,\,]}\bigl(r(\tau),\gamma(\phi_{\rm KZ})\bigr)$.

By the argument of Lemma \ref{sameoldthing}, the $\Q$-algebra generated by the
coefficients of $\mathfrak{b}(\tau)=\log_b\,\log\,\overline{B}(\tau)$ is
equal to the one generated by the coefficients of $\log\,\overline{B}(\tau)$,
which in turn is equal to the one generated by the coefficients of 
$\overline{B}(\tau)$, namely $\overline{\Bc}\subset
\Ec^{\rm geom}\otimes \overline{\Zc}$.  In order to show that these two
algebras are equal, we will consider $\overline{\Bc}$ as the $\Q$-algebra
generated by the coefficients of $\mathfrak{b}(\tau)$, and use properties 
of the right-hand side of \eqref{frakbtau} to show separately 
that it contains $\overline{\Zc}$ and $\Ec^{\rm geom}$.

\vspace{.3cm}
Let us write ${\Zc}_{>0}$ for the vector subspace of $\Zc$ spanned by 
all the multizeta values $\zeta(k_1,\ldots,k_r)$ with $r\ge 1$, and let
$\overline{\Zc}_{>0}$ denote the image of this space under the quotient
map $\Zc\rightarrow\!\!\!\!\!\rightarrow \overline{\Zc}$.
Then $\overline{\Zc}$ is spanned by $\Q$ and $\overline{\Zc}_{>0}$.
Let us write $\nz$ for the vector space quotient 
$\overline{\Zc}_{>0}/\overline{\Zc}_{>0}^2$, where $\overline{\Zc}_{>0}^2$ denotes
the vector subspace of $\overline{\Zc}$ generated by products of elements of
$\overline{\Zc}_{>0}$, which can be viewed as linear combinations thanks to the
shuffle multiplication of multizetas.  The vector space $\nz$ is called the
space of {\it new multiple zeta values} \cite{Furusho:Stablederivation}, \cite{Schneps:Emzv}.

Let $\cMZ$ denote the $\Q$-algebra of {\it motivic multiple zeta 
values} defined by Goncharov (see \cite{Goncharov:Galois}); let $\cMZ_{>0}$ denote the $\Q$-vector subspace generated by the motivic multizeta values,
and let $\nmz=\cMZ_{>0}/\cMZ_{>0}^2$ be the space of {\it
new motivic multizeta values}.  Goncharov showed that $\cMZ$ is a
Hopf algebra, so that $\cNMZ$ is a Lie co-algebra.
He further showed that the motivic $\zeta^{\mathfrak{m}}(2)=0$ in 
$\cMZ$, that there is a surjection 
$\cMZ\rightarrow\!\!\!\!\!\rightarrow\overline{\Zc}$, and that the
motivic multizeta values satisfy the associator relations of
$\overline{\Phi}_{\rm KZ}$. Letting $\nmz^{\vee}$ be the graded dual of $\nmz$, it follows that we have an inclusion of Lie algebras
\begin{equation}\label{inclusions}
\nmz^{\vee}\subset \grt
\end{equation}
and that the Lie series $\phi_{\rm KZ}$ lies in the Lie algebra
$\nmz^{\vee}\otimes_\Q \overline\Zc$.
Thus $\gamma(\phi_{\rm KZ})\in \gamma(\nmz^{\vee})\otimes_\Q\overline{\Zc}$ and
$$\gamma(\phi_{\rm KZ})\cdot b\in v_b\bigl(\gamma(\nmz^{\vee})\bigr)\otimes_\Q
\overline{\Zc}.$$ 
An important theorem by Brown (\cite{Brown:MTM}) identifies the Lie algebra $\nmz^{\vee}$ with the fundamental Lie algebra of the category of mixed Tate motives over $\Z$, which is
free on one generator in each odd weight $\geq 3$ (the weight used here corresponds to
the degree of the Lie polynomials, and is the negative of the usual motivic weight).

Let us introduce the derivations $\delta_{2k}$, defined by 
\begin{equation}\label{deltas}
\delta_{2k}([a,b])=0,\ \ \ \delta_{2k}(a) =\ad(a)^{2k}(b)\ \ {\rm for\ } k\ge 0.
\end{equation}
Similarly to $\fu$ and $\fu'$, let $\fv$ denote the completed Lie algebra
generated by the $\delta_{2k}$, $k\ge 0$, and $\fv'$ the completed Lie
subalgebra generated by the $\ad(\delta_0)^i(\delta_{2k})$ for $k\ge 1$,
$0\le i\le 2k-2$.

In \cite{HM}, Hain and Matsumoto defined a category of universal mixed 
elliptic motives, and they showed that the (completed) fundamental Lie algebra 
of that category has a monodromy representation in $\Der_0(\ff_2)$ whose image 
$\Pi$ is isomorphic to a semi-direct product $\Pi\cong \fv'\rtimes
\gamma(\nmz^{\vee})$.\footnote{The (non-completed) Lie algebra $\fv'$ is denoted
$\fu^{\rm geom}$ in \cite{HM}}  (To see that the section introduced in 
\cite{HM} coincides 
with $\gamma$ see \cite{Schneps:Emzv}.) Thus for every $D_1\in \fv'$ 
and $D_2\in \gamma(\nmz^{\vee})$, any bracketed word in these two derivations 
lies in $\fv'$.  Hain and Matsumoto further show that the Lie algebra
$\fv'\rtimes \gamma(\nmz^{\vee})$ is an $\mathfrak{sl}_2$-module, where
$\mathfrak{sl}_2$ is generated by $\eps_0$ and $\delta_0$.

By the representation theory of $\mathfrak{sl}_2$, both $\fu'$ and $\fv'$
are $\mathfrak{sl}_2$-modules, and indeed we have the equality of
Lie algebras 
\begin{equation}\label{uequalsv}
\fu'=\fv'\subset \Der_0(\ff_2).
\end{equation}

\begin{lem}\label{semidir} 
The Lie algebra generated inside $\Der_0(\ff_2)$ by
$\fu$ and $\gamma(\nmz^{\vee})$ has
the structure of a semi-direct product $\fu\rtimes\gamma(\nmz^{\vee})$.
\end{lem}

\begin{proof} As noted above, Hain and Matsumoto showed that 
$\fv'\rtimes \gamma(\nmz^{\vee})$ is an $\mathfrak{sl}_2$-module, so
by \eqref{uequalsv}, $\fu'\rtimes \gamma(\nmz^{\vee})$ is also an
$\mathfrak{sl}_2$-module. On the other hand, $\mathfrak{sl}_2$ is contained in the 
Lie subalgebra $\mathfrak{r}_{ell}$
of $\grt_{ell}$ defined by Enriquez in \cite{Enriquez:EllAss}.
The Lie algebra generated by
$\fu$ and $\gamma(\nmz^{\vee})$ is obtained by adjoining $\eps_0$ to 
$\fu'\rtimes \gamma(\nmz^{\vee})$.  To see that $\fu$ and $\gamma(\nmz^{\vee})$
form a semi-direct product, given that $\fu'$ and $\gamma(\nmz^{\vee})$
already do, we check separately that $[\eps_0,\fu']\subset\fu'$ (which is
clear since $\fu'$ is an $\mathfrak{sl}_2$-module) and that 
$[\eps_0,\gamma(\nmz^{\vee})]\subset \fu'$.  For this, we note that
$\eps_0$ lies in the Lie subalgebra $\mathfrak{r}_{ell}
\subset \grt_{ell}$, and $\grt_{ell}$ is isomorphic to the
semi-direct product $\mathfrak{r}_{ell}\rtimes \gamma(\grt)$ by \cite{Enriquez:EllAss}. Therefore since $\gamma(\nmz^{\vee})\subset\gamma(\grt)$, we must have
$[\eps_0,\gamma(\nmz^{\vee})]\subset \mathfrak{r}_{ell}$.  However, since
$\fu'\rtimes\gamma(\nmz^{\vee})$ is an $\mathfrak{sl}_2$-module, we also have
$[\eps_0,\gamma(\nmz^{\vee})]\subset \fu'\rtimes\gamma(\nmz^{\vee})$.  Thus
$[\eps_0,\gamma(\nmz^{\vee})]$ lies in the intersection of
$\fu'\rtimes\gamma(\nmz^{\vee})$ with $\mathfrak{r}_{ell}$, which is nothing but
$\fu'$.
\end{proof}

\vspace{.3cm}
In our situation, we take the two derivations
$$r(\tau)\in \fu\otimes_\Q \Ec^{\rm geom}\ \ {\rm and}\ \  \gamma(\phi_{\rm KZ}) \in \gamma(\nmz^{\vee})\otimes_\Q \overline{\Zc}.$$ 
By Lemma \ref{semidir}, any bracketed word combining these two derivations lies in 
the space $\bigl(\fu'\otimes_\Q \Ec^{\rm geom}
\otimes_\Q \overline{\Zc}\bigr)$; indeed, since a bracket of derivations
cannot be of degree $0$, it must lie in $\fu'$.  Thus in particular, we have 
$s(\tau)\in \fu'\otimes_\Q \bigl(\Ec^{\rm geom}\otimes_\Q \overline{\Zc}\bigr)$
and $s(\tau)\cdot b\in v_b(\fu')\otimes_\Q \bigl(\Ec^{\rm geom}\otimes_\Q \overline{\Zc}\bigr)$, where
$s(\tau)$ is the sum of bracketed terms in $ch_{[\,,\,]}\bigl(r(\tau),\gamma(\phi_{\rm KZ})\bigr)$ as in \eqref{frakbtau}. Altogether, we thus have
\begin{equation}\label{threecases}
\begin{cases} \gamma(\phi_{\rm KZ})\cdot b\in v_b\bigl(\gamma(\nmz^{\vee})\bigr)\otimes_\Q \overline\Zc\\
r(\tau)\cdot b\in v_b(\fu)\otimes_\Q \Ec^{\rm geom}\\
s(\tau)\in v_b(\fu')\otimes_\Q \bigl(\Ec^{\rm geom}\otimes_\Q \overline\Zc\bigr)
\end{cases}
\end{equation}
for the three terms in the right-hand side of \eqref{frakbtau}. 

We are now ready to show that $\overline{\Bc}\supset \overline{\Zc}$. 
Since the evaluation map $v_b$ is injective on $\Der_0(\ff_2)$, it is in 
particular injective on
the subspace $\fu\rtimes \gamma(\nmz^{\vee})$.  Let
$V$ denote the underlying vector space of 
$v_b(\fu)$, and $W$ that of $v_b\bigl(\gamma(\nmz^{\vee})\bigr)$.
Then the underlying vector space of the semi-direct product
$v_b\bigl(\fu\rtimes \gamma(\nmz^{\vee})\bigr)$ is the direct sum
$V\oplus W$. Writing $R=\Ec^{\rm geom}\otimes_\Q \overline{\Zc}$, we deduce
from \eqref{threecases} that
\begin{equation}\label{threecasesbis}
\gamma(\phi_{\rm KZ})\cdot b\in W\otimes_\Q R\ \ \ {\rm and}\ \ \ 
r(\tau)\cdot b,\ s(\tau)\cdot b\in V\otimes_\Q R.
\end{equation}

Let us take a linear basis of $V\oplus W$ adapted to the direct product, that is, in
which every element belongs either to $V$ or to $W$.  Write
$\mathfrak{b}(\tau)$ in this basis, and consider the coefficient of a 
basis element $w\in W$.  By (\ref{threecasesbis}), 
$\mathfrak{b}(\tau)$ decomposes as a sum
of two terms, $\gamma(\phi_{\rm KZ})\cdot b\in W\otimes_\Q R$ and
$r(\tau)\cdot b+s(\tau)\cdot b\in V\otimes_\Q R$.  Therefore, 
the coefficient in $\mathfrak{b}(\tau)$ of any basis element $w\in W$ 
is equal to the coefficient of $w$ in $\gamma(\phi_{\rm KZ})\cdot b$ written
in the same basis of $W$.
However, since we know that the coefficients of $\phi_{\rm KZ}$ written in
a basis of $\nmz^{\vee}$ multiplicatively generate all of $\overline{\Zc}$,
and since $\gamma$ is injective and defined over $\Q$, the same holds
for the coefficients of $\gamma(\phi_{\rm KZ})$ written in a basis of
$\gamma(\nmz^{\vee})$, and then since $v_b$ is injective on this space and
defined over $\Q$, the same again holds for the coefficients of 
$v_b\bigl(\gamma(\phi_{\rm KZ})\bigr)$ written in a basis of $W$.
Thus the coefficients in $\mathfrak{b}(\tau)$ of the elements of the
basis of $W$ span all of $\overline{\Zc}$, so 
$\overline{\Bc}\supset \overline{\Zc}$.

Now we will show that $\overline{\Bc}\supset \Ec^{\rm geom}$. Here we need
to deal with the $s(\tau)$ term.  For this, we will proceed by induction on
the weight.  We take a basis for $V$ which is the image under $v_b$ of 
a weight-graded basis of $\fu$.  The Lie series $\mathfrak{b}(\tau)$ starts 
in weight 1 with the term $2\pi i\tau a$, which comes from the 
$2\pi i\tau\eps_0$ term of $r(\tau)$ acting on
$b$.  Since $r'(\tau):=r(\tau)-2\pi i\tau\eps_0$ and $s(\tau)$ lie in $\fu' \otimes R$, these derivations are 
strictly weight-increasing, so there are no other weight 1 terms in $\mathfrak{b}(\tau)$.  Thus $2\pi i\tau \in\overline{\Bc}$. We use this result as the
base case, fix $n>1$, and make the induction hypothesis 
that for all $m<n$, the coefficients in $\mathfrak{b}(\tau)$ of the 
weight $m$ basis elements of $V$ span the weight $m$ graded part of 
$\Ec^{\rm geom}$, so that $\overline{\Bc}$
contains the weight graded parts of $\Ec^{\rm geom}$ for all weight $m<n$.  
Consider the coefficient in $\mathfrak{b}(\tau)$ of
a weight $n$ basis element $v\in V$.  Each coefficient is of the form
$r_v+s_v$ where $r_v$ is the coefficient of $v$ in $r(\tau)\cdot b$ and
$s_v$ is the coefficient of $v$ in $s(\tau)\cdot b$. However, the part of 
the derivation $s(\tau)$ that takes $b$ to a weight $n$ polynomial is
made up of brackets of parts of $r(\tau)$ and of $\gamma(\phi_{\rm KZ})$ of
strictly smaller weight, and whose coefficients are thus algebraic
combinations of coefficients of $r(\tau)$ of lower weight, 
which already appear in $\overline{\Bc}$ by the induction hypothesis, and 
of coefficients of $\gamma(\phi_{\rm KZ})$, that is, elements of $\overline{\Zc}$, 
which already lie in $\overline{\Bc}$ by the result above that 
$\overline{\Bc}\supset \overline{\Zc}$.  Thus 
not just the coefficient $r_v+s_v$, but also the term
$s_v$ lies in $\overline{\Bc}$, which proves that $r_v\in \overline{\Bc}$.
Thus all the coefficients of the weight $n$ part of $r(\tau)\cdot b$ lie in
$\overline{\Bc}$, so by induction, all the coefficients of $r(\tau)\cdot b$
lie in $\overline{\Bc}$; since $v_b$ is injective, these coefficients 
generate the same $\Q$-algebra as the coefficients of $r(\tau)$, namely
$\Ec^{\rm geom}$.  This proves that $\overline{\Bc}\supset \Ec^{\rm geom}$,
and completes the proof of the desired result $\overline{\Bc}\simeq \Ec^{\rm geom}\otimes_\Q \overline{\Zc}$. 

\vspace {.2cm}
\noindent {\bf Step 2. The case of $\overline{\Ec}$.} The argument is
similar to the one for $\overline{\Bc}$, but there is an added
subtlety coming from the fact that $r(\tau)$ lies in $\Der_0(\ff_2)$,
but $v_a$ is not injective on $\Der_0(\ff_2)$ since $\eps_0(a)=0$.  To
get around this, we will use the fact that 
$$\fu\simeq \fu'\rtimes \Q\eps_0$$
(see \eqref{fuprod}).  

It is known that the underlying vector space of the universal
enveloping algebra $\U(\mathfrak{g}\rtimes \mathfrak{h})$ of a semi-direct product of graded
$\Q$-Lie algebras $\mathfrak{g}\rtimes \mathfrak{h}$ is the space $\U(\mathfrak{g})\otimes_\Q\U(\mathfrak{h})$; in fact
$\U(\mathfrak{g}\rtimes \mathfrak{h})$ is a Hopf algebra equipped with the smash product 
(\cite{Molnar:Hopf}) and with the standard coproduct for which elements of 
$\mathfrak{g}\rtimes \mathfrak{h}$ are primitive. The graded dual $\U(\mathfrak{g}\rtimes \mathfrak{h})^\vee$ 
has underlying $\Q$-algebra $\U(\mathfrak{g})^\vee\otimes_\Q\U(\mathfrak{h})^\vee$ 
(and is equipped with the smash coproduct).  
Using this, we see that as $\Q$-algebras,
$$\mathcal{U}(\fu)^\vee\simeq \mathcal{U}(\fu')^\vee\otimes_\Q \mathcal{U}(\Q\eps_0)^\vee.$$
Under the identification $\mathcal{U}(\fu)^\vee\simeq \Ec^{\rm geom}$ of
Theorem \ref{duality}, this translates to
\begin{equation}\label{E0geom}
\Ec^{\rm geom}\simeq \Ec_0^{\rm geom}\otimes_\Q \Q[2\pi i\tau],
\end{equation}
where $\Ec_0^{\rm geom}$ is multiplicatively generated by the coefficients
of $r'(\tau)=r(\tau)-2\pi i\tau\eps_0$.  In particular, the subspace
inclusion $\fu'\subset \fu$ corresponds in the dual to the surjection
$$\Ec^{\rm geom}\rightarrow\!\!\!\!\!\rightarrow \Ec^{\rm geom}/\langle 2\pi i\tau\rangle\simeq \Ec_0^{\rm geom}.$$
The derivation $\delta(\tau)=ch_{[\,,\,]}\bigl(r(\tau),\gamma(\phi_{\rm KZ})\bigr)$
lies in $\fu\otimes_\Q \bigl(\Ec^{\rm geom}\otimes_\Q \overline{\Zc}\bigr)$, but if we
consider the derivation $\hat\delta(\tau)$ obtained by reducing its coefficients
mod $2\pi i\tau$, then we see that 
$$\hat\delta(\tau)=ch_{[\,,\,]}\bigl(r'(\tau),\gamma(\phi_{\rm KZ})\bigr)\in\fu'\otimes_\Q 
\bigl(\Ec_0^{\rm geom}\otimes_\Q \overline{\Zc}\bigr),$$
where $r'(\tau)=r(\tau)-2\pi i\tau\eps_0$.

Let $\widehat{\overline{E}}(\tau)$ and $\hat{\mathfrak{e}}(\tau)$ be
the power series obtained from $\overline{E}(\tau)$ and $\mathfrak{e}(\tau)\in
\ff_2\otimes_\Q \bigl(\Ec^{\rm geom}\otimes_\Q \overline{\Zc}\bigr)$ by
reducing the coefficients mod $2\pi i\tau$.  Then we have
$$\widehat{\overline{E}}(\tau),\ \hat{\mathfrak{e}}(\tau)\in \ff_2\otimes_\Q
\bigl(\Ec_0^{\rm geom}\otimes_\Q \overline{\Zc}\bigr)$$
and
$$\widehat{\overline{E}}(\tau)=\exp\bigl(\hat\delta(\tau)\bigr)\cdot a,\ \ \hat{\mathfrak{e}}(\tau)=\hat\delta(\tau)\cdot a,$$
so
$$\hat{\mathfrak{e}}(\tau)=\log_a\,\widehat{\overline{E}}(\tau).$$
Thus the $\Q$-algebras generated by the coefficients of $\hat{\mathfrak{e}}(\tau)$
and by $\widehat{\overline{E}}(\tau)$ are equal.  Denote this $\Q$-algebra
by $\mathfrak{E}\subset \Ec_0^{\rm geom}\otimes_\Q \overline{\Zc}$.  Since
$$\hat{\mathfrak{e}}(\tau)=ch_{[\,,\,]}\bigl(r'(\tau),\gamma(\phi_{\rm KZ})\bigr)\cdot a=r'(\tau)\cdot a+\gamma(\phi_{\rm KZ})\cdot a+\hat{s}(\tau)\cdot a$$
where $\hat{s}(\tau)$ denotes the bracketed terms in the Campbell-Hausdorff
product, we can use the identical arguments to the case of $\mathfrak{b}(\tau)$
above to prove that $\mathfrak{E}\supset \overline{\Zc}$.  We also again use
induction on the weight to prove that $\mathfrak{E}$ contains all the 
coefficients of $r'(\tau)$.  The only difference with the case of $\mathfrak{b}(\tau)$ is the base case, which is no longer in weight 1. The lowest weight term
of $\hat{\mathfrak{e}}(\tau)$ is of weight 3, and it comes from the term 
$\Gc_2(\tau)\eps_2$ of $r'(\tau)$ acting on $a$ (note that $\hat{s}(\tau)\cdot a$ has no
terms of weight lower than 7).  Thus the same induction as above works to prove
that every coefficient of $r'(\tau)$ lies in $\mathfrak{E}$, so we find that
$\mathfrak{E}\simeq \Ec_0^{\rm geom}\otimes_\Q \overline{\Zc}$.

Since $\mathfrak{E}$ is the $\Q$-algebra generated by the reduction of
$\overline{E}(\tau)$ mod $2\pi i\tau$ and it is equal to
$\Ec_0^{\rm geom}\otimes_\Q \overline{\Zc}$, we see that 
$\mathfrak{E}[2\pi i\tau]\simeq \Ec^{\rm geom}\otimes_\Q \overline{\Zc}$.
This means that the composition map
$$\overline{\Ec}\hookrightarrow \Ec^{\rm geom}\otimes_\Q \overline{\Zc}
\rightarrow\!\!\!\!\!\rightarrow 
\bigl(\Ec^{\rm geom}\otimes_\Q \overline{\Zc}\bigr)/
\langle 2\pi i\tau\otimes 1\rangle\simeq \Ec_0^{\rm geom}\otimes_\Q 
\overline{\Zc}$$
is surjective, which gives us the desired equality
$\overline{\Ec}[2\pi i\tau] \simeq \Ec^{\rm geom}\otimes_\Q \overline{\Zc}.$

\vspace{.1cm}
\noindent {\bf Step 3. The case of $\overline{\Ac}$.}  The argument here is
identical to the one for $\overline{\Ec}$.  We again have the problem that,
as discovered independently by Enriquez and by Hain and Matsumoto, there
exists a unique (up to scalars) derivation $\eta\in \fu$ that annihilates $t_{01}$;
$\eta$ is an infinite series in the 
$\eps_{2k}$, $k\ge 0$, with rational coefficients (\cite{Enriquez:EllAss}, proof of Proposition 6.3; \cite{HM}, Equation (27.3))
$$\eta=\sum_{k=0}^{\infty}(2k-1)\frac{B_{2k}}{(2k)!}\eps_{2k}$$
The exact nature of $\eta$ is not important, only the fact that it has a
term in $\eps_0$ and is defined over $\Q$; this shows that
$v_{t_{01}}$ is injective on $\fu'$.  We let $W=v_{t_{01}}(\nmz^{\vee})$
and $V=v_{t_{01}}(\fu')$, and choose a basis of $V$ which is the image under
$v_{t_{01}}$ of a weight-graded basis of $\fu'$ as above.  We then proceed 
exactly as in the case of $\overline{\Ec}$ to show that the $\Q$-algebra 
generated by the coefficients of $\exp\bigl(\hat\delta(\tau)\bigr)\cdot t_{01}$,
(which is the reduction of $\log\,\overline{A}(\tau)$ mod $2\pi i\tau$) contains $\overline{\Zc}$.  For the induction argument, even 
if $V$ is not itself graded by the weight, we can simply transport the 
weight-grading of $\fu'$ to $V$ and use induction on that (or equivalently,
do the induction on the lowest weight parts of the basis elements).
Since $t_{01}$ starts with $-a$, the argument is identical to the one for
$a$ above, and shows that the $\Q$-algebra generated by the coefficients
of $\hat\delta(\tau)\cdot t_{01}$, and thus also by those of
$\exp\bigl(\hat\delta(\tau)\bigr)\cdot t_{01}$, is isomorphic to
$\Ec_0^{\rm geom}\otimes_\Q \overline{\Zc}$, so that again we have
$\overline{\Ac}[2\pi i\tau]\simeq \Ec^{\rm geom}\otimes_\Q \overline{\Zc}$ as desired.
This concludes the proof.
\end{proof}

\section{The elliptic double shuffle and push-neutrality relations} 
In this section we use mould theory to explore and compare algebraic 
relations between the $\overline{E}$-EMZs with algebraic relations
between the $\overline{A}$-EMZs.  The first paragraph, \S 4.1, gives
a brief exposition of the necessary definitions and results from mould theory.

Our main result on $\overline{E}$-elliptic multizetas, in \S 4.2, arises as a corollary of 
the preceding theorem and the main result of \cite{Schneps:Emzv}. We show that
$\overline E(\tau)$ satisfies a certain double family of algebraic
relations called the {\it elliptic double shuffle relations},
related to the familiar double shuffle properties of
$\Phi_{\rm KZ}$, but more similar to the graded double shuffle relations studied 
for example in \cite{Brown:depth}.  Further, we show that if
one assumes certain reasonable conjectures from multizeta and 
Grothendieck-Teichm\"uller theory, the elliptic double shuffle 
relations form a {\it complete} set of algebraic relations 
for the $\overline{E}$-EMZs. We compute these relations and the
associated dimensions in detail in depth 2.

Finally, in \S 4.3 we consider a double family of relations satisfied by
$\overline{A'}(\tau)$ (or more precisely by the log of this series).
The first family is just the usual shuffle, but the second is very different 
from the second shuffle relation satisfied by 
$\overline{E}(\tau)$. We call it the family of {\it push-neutrality relations},
and show that it is related to the  \textit{Fay relations} studied in
\cite{Matthes:Edzv}.  We compute the relations and the associated dimensions
in depth 2 and show that they are different from those of $\overline{E}(\tau)$,
which means that while we know by Theorem \ref{tensprod} that
$\overline{\Ec}[2\pi i\tau]=\overline{\Ac}[2\pi i\tau]$,
the algebras $\overline{\Ec}$ and $\overline{\Ac}$ themselves are not equal
nor even isomorphic as filtered algebras (that is, the dimensions of the
associated gradeds are not equal).

\subsection{A very brief introduction to moulds} \label{ssec:moulds}

We recall some notions from Ecalle's theory of moulds \cite{Ecalle:Dimorphie,Ecalle:Flexion} that we will need in order to study algebraic relations between elliptic multizetas. Besides the original references, a more detailed introduction to moulds can be found in \cite{Schneps:ARI}.

\subsubsection{Moulds and bialternality}
In this article, we use the term `mould' to refer only
to rational-function valued moulds with coefficients in $\Q$. Thus, 
a mould is a family of functions 
$$\{P(u_1,\ldots,u_r)\mid r\ge 0\}$$
with $P(u_1,\ldots,u_r)\in \Q(u_1,\ldots,u_r)$.  In particular
$P(\emptyset)$ is a constant.  The {\it depth $r$} part of a mould is
the function $P(u_1,\ldots,u_r)$ in $r$ variables. By defining addition
and scalar multiplication of moulds in the obvious way, that is, depth by
depth, moulds form a $\Q$-vector space that we call $Moulds$.  We
write $Moulds_{pol}$ for the subspace of polynomial-valued moulds.
The vector space $ARI$ is the subspace of $Moulds$ consisting of
moulds $P$ with constant term $P(\emptyset)=0$, and $ARI_{pol}$ is
again the subspace of polynomial-valued moulds in $ARI$.

The standard mould multiplication $mu$ is given by
\begin{equation}\label{mumould}
mu(P,Q)(u_1,\ldots,u_r)=\sum_{i=0}^r P(u_1,\ldots,u_i)Q(u_{i+1},\ldots,u_r).
\end{equation}
For simplicity, we write $P\,Q=mu(P,Q)$.  This multiplication
defines a Lie algebra structure on $ARI$ with Lie bracket
$lu$ defined by $lu(P,Q)=mu(P,Q)-mu(Q,P)$.

We now introduce four operators on moulds.  The $\Delta$-operator on moulds
is defined as follows: if $P\in ARI$, then
\begin{equation}\label{Deltamould}
\Delta(P)(u_1,\ldots, u_r) =  u_1\cdots u_r (u_1+\cdots+u_r) P(u_1,\ldots,u_r).
\end{equation}
The $dar$-operator is defined by
\begin{equation}\label{darmould}
dar(P)(u_1,\ldots, u_r) =  u_1\cdots u_r\, P(u_1,\ldots,u_r).
\end{equation}
The $push$-operator is defined by
\begin{equation}\label{pushmould}
push(B)(u_1,\ldots,u_r)=B(u_2,\ldots,u_r,-u_1-\cdots-u_r).
\end{equation}
Finally, the {\it swap} operator is defined by
\begin{equation}\label{swapmould}
swap(A)(v_1,\ldots,v_r)=A(v_r,v_{r-1}-v_r,\ldots,v_1-v_2).
\end{equation}
Here the use of the alphabet $v_1,v_2,\ldots$ instead of $u_1,\ldots,u_r$
is purely a convenient way to distinguish a mould from its swap.

The main property on moulds that we will need to consider is {\it alternality}.
A mould $P$ is said to be {\it alternal} if for all $r>1$ and
for $1\le i\le [r/2]$, we have
\begin{equation}\label{alternal}
\sum_{{\bf u}\in sh((u_1,\ldots,u_i),(u_{i+1},\ldots,u_r))} P({\bf u})=0,
\end{equation}
where the set of $r$-tuples $sh\bigl((u_1,\ldots,u_i),(u_{i+1},\ldots,
u_r)\bigr)$ is the set 
$$\bigl\{(u_{\sigma^{-1}(1)},\ldots,u_{\sigma^{-1}(r)})\,\bigl|\,
\sigma\in S_r\ \ {\rm such\ that}\ 
\sigma(1)<\cdots<\sigma(i),\ \  \sigma(i+1)<\cdots<\sigma(r)\bigr\}.$$

The mould $swap(A)$ is alternal if it satisfies the same property 
\eqref{alternal} in the variables $v_i$.

We write $ARI^{al}$ for the space of alternal moulds in $ARI$, and
$ARI^{al/al}$ for the space of moulds which are alternal and whose
swap is also alternal.  We also consider moulds which are alternal
and whose swap is alternal up to addition of a constant-valued mould. 
The space of these moulds is denoted $ARI^{al*al}$ and we call them
{\it bialternal}.  

\subsubsection{From power series to moulds}
Let $c_i=\ad(a)^{i-1}(b)$ for $i\ge 1$ as in \S 3.1.  Let the depth of a 
monomial $c_{i_1}\cdots c_{i_r}$ be the number $r$ of $c_i$ in the monomial;
the weight (degree in $a$ and $b$) and the depth 
form a topological bigrading on the formal power series ring 
$\Q\langle\!\langle C\rangle\!\rangle =
\Q\langle\!\langle c_1,c_2,\ldots\rangle\!\rangle$ on the free variables $c_i$. Here, by ``topological bigrading'' we mean that $\Q\langle\!\langle C\rangle\!\rangle$ is the direct product (not the direct sum) $\prod_{n,d\geq 0}V_{n,d}$ of its components of weight $n$ and depth $d$.
Similarly, we write
$L[\![C]\!]=\Lie[\![c_1,c_2,\ldots]\!]$
for the corresponding Lie algebra. By Lazard 
elimination, we have an isomorphism
\begin{equation*}\label{Qc}
\Q a\oplus L[\![C]\!]\cong \ff_2=\Lie[\![a,b]\!].
\end{equation*}

Following \'Ecalle, let $ma$ denote the standard vector space isomorphism from 
$\Q\langle\!\langle C\rangle\!\rangle$ to the space 
$(Moulds)^{pol}$ defined by
\begin{align}ma:\Q\langle\!\langle C\rangle\!\rangle&\buildrel\sim\over\rightarrow (Moulds)^{pol}\notag\\
c_{k_1}\cdots c_{k_r}&\mapsto (-1)^{k_1+\cdots+k_r-r}
u_1^{k_1-1}\cdots u_r^{k_r-1}\label{ma}
\end{align}
on monomials, extended by linearity to all power series.

It is well known that $p\in\Q\langle\!\langle C\rangle\!\rangle$ satisfies the shuffle relations 
if and only if $p$ is a Lie series, that is, $p\in \Lie[\![C]\!]$.  
The alternality property on moulds is analogous to these shuffle
relations, that is a series $p\in\Q\langle\!\langle C\rangle\!\rangle$ satisfies the shuffle 
relations if and only if $ma(p)$ is alternal (see e.g. \cite{Schneps:ARI}, \S 2.3 and Lemma 3.4.1]).
Writing $ARI^{al}$ for the subspace of alternal moulds and 
$ARI_{pol}^{al}$ for the subspace of alternal polynomial-valued moulds, this
shows that the map $ma$ restricts to a Lie algebra isomorphism 
$$ma:\Lie[\![C]\!]\buildrel{ma}\over\longrightarrow ARI_{lu,pol}^{al}.$$ 

We now recall that for any mould $P\in ARI$, \'Ecalle defines
a derivation $arit(P)$ of the Lie algebra $ARI_{lu}$.  We do not
need to recall the definition of $arit$ here (but it is given in
\S 4.4 below where we prove a technical lemma). For now it is enough to 
know that when restricted to polynomial-valued moulds, it is related to the 
Ihara derivations \eqref{iharader} via the morphism $ma$:
$$ma\bigl(D_f(g)\bigr)=-arit\bigl(ma(f)\bigr)\cdot ma(f).$$
For each $P\in ARI$, we also define the derivation
\begin{equation}\label{aratdef}
arat(P)=-arit(P)+\ad(P),
\end{equation}
where $\ad(P)\cdot Q=lu(P,Q)$.  
\vskip .2cm
We recall from \S 3.1 of \cite{Schneps:Emzv} the definition
$$Darit(P)=-dar\Bigl(arit\bigl(\Delta^{-1}(P)\bigr)-ad\bigl(\Delta^{-1}(P)\bigr)\Bigr)\circ dar^{-1}.$$
While complicated, it is shown in \cite{Schneps:Emzv} that for any mould $P\in ARI$,
$Darit(P)$ defines a derivation on $ARI$ for the $lu$-bracket that annihilates the
mould $-u_1=ma([a,b])$, and that the derivation $Darit(P)$ can be considered as taking
mould-values on $a$ and $b$, namely $Darit(P)\cdot a=P$ and $Darit(P)\cdot b$ is by
definition $Darit(P)\cdot ma(b)$, where $ma(b)$ is the mould that takes value $1$ in
depth 1 and 0 in all other depths.  If $P=ma(p)$ for a push-invariant Lie series
$p(a,b)$, then $q=Darit(P)\cdot b$ is the unique Lie power series such that the
derivation $a\mapsto p$, $b\mapsto q$ annihilates $[a,b]$.

It is shown in Proposition 3.2.1 of \cite{Schneps:Emzv} that the transport of the $ari$-bracket on $ARI$ under the vector space isomorphism $\Delta$ is equal to the $Dari$-bracket
$$Dari(P,Q)=Darit(P)\cdot Q-Darit(Q)\cdot P.$$
Restricted to the subspace $ma(\ff_2^{\rm push})$, this bracket extends the 
$\langle \,,\,\rangle$ bracket on $\ff_2^{\rm push}$ defined in \eqref{expa}.
We can define a  exponential map $\exp_{Dari}:ARI\rightarrow GARI$ by
$$\exp_{Dari}(P)=1+P+Darit(P)(P)+{\frac{1}{2}}Darit(P)^2(P)+\cdots,$$
and put a new corresponding group law $Dgari$ on $GARI$ by
$$Dgari\bigl(exp_{Dari}(P),exp_{Dari}(Q)\bigr)=exp_{Dari}\bigl(ch_{Dari}(P,Q)\bigr).$$
There is a unique group homomorphism $\Delta^*:GARI_{gari}\rightarrow GARI_{Dgari}$
that makes the following diagram commute.
\begin{equation}\label{deltastar}
\xymatrix{GARI\ar[r]^{\Delta^*}&GARI\\
ARI\ar[u]^{\exp_{ari}}\ar[r]^{\Delta}&ARI\ar[u]_{\exp_{Dari}}.}
\end{equation}

\vspace{.3cm}
To summarize, 
\vskip .2cm\noindent $\bullet$
The $mu$-multiplication extends multiplication on power series to all moulds;
\vskip .2cm\noindent $\bullet$
The $lu$-bracket extends the usual bracket $[\,,\,]$ on $\ff_2$ to all moulds in $ARI$;
\vskip .2cm\noindent $\bullet$
The $ari$-bracket extends the twisted Magnus (or Ihara) bracket on the underlying
vector space of $\ff_2$ to all moulds in $ARI$;
\vskip .2cm\noindent $\bullet$
The $arit(F)$ derivations for $F\in ARI$ extend Ihara's derivations $D_f$ 
for $f\in \ff_2$;
\vskip .2cm\noindent $\bullet$
The alternality property extends to all moulds the shuffle property on
power series (that is, the property that says a power series in $a$ and $b$ is a Lie series).
\vskip .2cm\noindent $\bullet$
The bialternality property extends to all moulds the
linearized double shuffle property defining the linearized double shuffle space $\ls\subset
\Q\langle\!\langle a,b\rangle\!\rangle$.

\vspace{.3cm}
We end this subsection by adding to this list the group-like version of these notions.
\vskip .2cm\noindent $\bullet$
The group $GARI$ of moulds having constant term $1$, equipped with the $mu$-multiplication,
extends the group of power series in $a,b$ with constant term 1.
\vskip .2cm\noindent $\bullet$
The $ari$-exponential $\exp_{ari}:ARI\rightarrow GARI$ extends the twisted Magnus
exponential $\exp^\odot:\ff_2\rightarrow F_2(\Q)$, and its inverse $\log_{ari}$
extends the twisted Magnus logarithm.
\vskip .2cm\noindent $\bullet$ 
A mould $A\in GARI$ is said to be {\it symmetral} if 
and only if $\log_{ari}(A)$ is alternal.  The symmetrality property
extends the property of being group-like for a power series with constant term 1. 
\vskip .2cm\noindent $\bullet$ 
The bisymmetrality property that a mould $A\in GARI$ is symmetral with symmetral
swap extends the group-like double shuffle relations on power series in $a$ and $b$.
A mould $A\in GARI$ is bisymmetral if and only if $\log_{ari}(A)$ is bialternal.
\vskip .2cm\noindent $\bullet$ 
The derivation $Darit(P)$ extends the derivation associated to $p\in \ff_2^{\rm push}$
that maps $a\mapsto p$ and annihilates $[a,b]$.
\vskip .2cm\noindent $\bullet$ 
The Lie bracket $Dari$ on $ARI$ extends the Lie bracket $\langle\,,\,\rangle$ on
$\ff_2^{\rm push}$ defined in \eqref{expa}.
\vskip .2cm\noindent $\bullet$ 
The exponential map $\exp_{Dari}:ARI\rightarrow GARI$ extends to all moulds in $ARI$
the exponential map $\exp_a$ on $\ff_2^{\rm push}$ defined in \eqref{diag0}, except that if $p\in \ff_2^{\rm push}$ and 
$P=ma(P)$, then we have
$$\exp_{Dari}(P)=1-a+ma\bigl(\exp_a(p)\bigr).$$

\subsubsection{Reminders on the elliptic double shuffle Lie algebra $\ds_{ell}$}

\vspace{.2cm}
The elliptic double shuffle relations are based on the mould properties
given in the following definition.

\begin{dfn}\label{deltabi}
A mould $P\in ARI$ is said to be $\Delta$-bialternal if $\Delta^{-1}(P)$ is
bialternal, and we write $ARI^{\Delta\hbox{-}al*al}$ for the space of
such moulds. Similarly, a mould $Q\in GARI$ is said to be $\Delta^*$-bisymmetral
if $(\Delta^*)^{-1}(Q)$ is bisymmetral, or equivalently (by \eqref{deltastar}), if 
$\log_{Dari}(Q)$ is $\Delta$-bialternal.
\end{dfn}

The {\it elliptic double shuffle Lie algebra $\ds_{ell}$} is defined as
follows in \cite{Schneps:Emzv}.

\begin{dfn}The {\it elliptic double shuffle Lie algebra}
$\ds_{ell}$ is the subspace of $\ff_2$ such that
\begin{equation}\label{dselldef}
ma\bigl(\ds_{ell}\bigr)=ARI_{pol}^{\Delta\hbox{-}al*al},
\end{equation}
that is, $\ds_{ell}$ consists of the $f\in \Q\langle\!\langle a,b\rangle\!\rangle$
such that $ma(f)$ is $\Delta$-bialternal.  
\end{dfn}

The main point about these definitions is that $\Delta$-alternality 
for Lie-like power series (and its group-like version 
$\Delta^*$-bisymmetrality for group-like power series), are in fact
``transports'' of the usual Lie-like double shuffle relations satisfied by
$\phi_{\rm KZ}$ (resp.~the group-like double shuffle relations satisfied
by $\Phi_{\rm KZ}$). In other words, the mould associated to 
$\gamma(\phi_{\rm KZ})$ is $\Delta$-bialternal (and the mould associated
to $\Gamma(\overline{\Phi_{\rm KZ}})=\overline{E}$ is $\Delta^*$-bisymmetral).
This result, together with the fact that these properties
are preserved by the action of $g(\tau)$, so that the same elliptic double 
shuffle relations are also satisfied by $\overline{E}(\tau)$ (group-like
version) and $\mathfrak{e}(\tau)$ (Lie-like version), forms the main theorem
of this section, which is stated and proved in the following subsection.
We first need to recall some known properties of the elliptic double shuffle
Lie algebra.

The following statements are essentially contained in \cite{BS} and 
\cite{Schneps:Emzv}. We give some details of the proofs for the convenience of 
the reader.  Recall from \eqref{expa} that $v_a$ gives an injective map of 
Lie algebras
$$\Der'_0(\ff_2)\rightarrow \ff_2^{\rm push},$$
where the right-hand Lie algebra is equipped with the bracket $\langle\ ,\ 
\rangle$ which is simply the transport of the
bracket of derivations by $v_a$.

\begin{prop}\label{dsell} The space $\ds_{ell}$ satisfies the following properties.
\begin{enumerate}
\item[(i)]
$\ds_{ell}\subset \ff_2^{\rm push}$.
\item[(ii)]
$\ds_{ell}$ is a Lie algebra under the bracket $\langle\,,\,\rangle$
on $\ff_2^{\rm push}$.
\item[(iii)]
There are Lie algebra injections
$$\fu'\rtimes \gamma(\nmz^{\vee})\subset \widetilde{\grt}_{ell}\hookrightarrow\ds_{ell},$$
where $\widetilde{\grt}_{ell}$ is the Lie subalgebra of
$\grt_{ell}$ generated by $\gamma(\grt)$ and $\fu'$, which is mapped into
$\ds_{ell}$ by $v_a$.
\end{enumerate}
\end{prop}
\begin{proof}
For (i), by definition, elements of $ma(\ds_{\rm ell})$ are of the form $\Delta(P)$ where
$P$ is an alternal mould. The map $\Delta$ trivially preserves alternality, therefore
$\Delta(P)\in \ds_{ell}$ is alternal, and since it is a polynomial mould, it corresponds
to a Lie series $p\in \ff_2$.  The mould $P$ is in fact bialternal, and it is shown in 
Lemma 2.5.5 of \cite{Schneps:ARI} that bialternal moulds are push-invariant.  The map 
$\Delta$ trivially preserves push-invariance, so $\Delta(P)$ is also push-invariant.
It is shown in Proposition 12 of \cite{RS} that for $p\in \ff_2$, $ma(p)$ is a 
push-invariant mould if and only $p\in \ff_2^{\rm push}$.
Thus $\ds_{\rm ell} \subset \ff_2^{\rm push}$.

For (ii), we saw in the previous subsection that the $Dari$-bracket 
extends the $\langle\,,\,\rangle$ bracket on $\ff_2^{\rm push}$ to all of $ARI$.
So to prove (ii), it is enough to prove that the subspace of $\Delta$-bialternal
moulds of $ARI$ is closed under the $Dari$-bracket.  Since the $\Dari$-bracket is
the transport of the $ari$-bracket by $\Delta$ (per the previous subsection),
this is equivalent to showing that the subspace of bialternal moulds in $ARI$
is closed under the $ari$-bracket. However, this is a well known result of \'Ecalle
(cf.~\cite{Schneps:ARI}, Thm. 2.5.6 for a complete proof.)

Finally, for (iii), we saw in \eqref{inclusions} that $\nmz^{\vee}\subset\grt$, so
$\gamma(\nmz^{\vee})\subset\gamma(\grt)$, which settles the first inclusion.
For the second
inclusion, we show separately that $\gamma(\grt)$ and $\fu'$ are both mapped into $\ds_{\rm ell}$. The first inclusion $v_a(\gamma(\grt))\subset \ds_{\rm ell}$ is given in
Theorem 1.3.1 of \cite{Schneps:Emzv}. 

For the inclusion $v_a(\fu') \subset \ds_{\rm ell}$, we use the Lie
algebra $\fv$ generated by the derivations $\delta_{2k}$ defined in \eqref{deltas}.
By \eqref{uequalsv}, we have $\fu'\subset \fv$.  We show that 
$v_a(\fv)\subset \ds_{ell}$.
For all $k\ge 0$, the mould $U_{2k}:=ma\bigl(\delta_{2k}(a)\bigr)$ is equal to 
$u_1^{2k}$ in depth 1 and is zero in all other depths, so the mould 
$\Delta^{-1}(U_{2k})$ is equal to $u_1^{2k-2}$ in depth 1 and zero in all 
other depths.  A mould concentrated in depth 1 is bialternal by default, 
so the $U_{2k}$ are all $\Delta$-bialternal; thus $v_a(\delta_{2k})=
\delta_{2k}(a)\in \ds_{ell}$ for all $k\ge 0$.  Since $\ds_{ell}\subset 
\ff_2^{\rm push}$ and the bracket $\langle \,,\,\rangle$ on $\ds_{ell}$ 
corresponds to the bracket of derivations, the values on $a$ of all brackets 
of the derivations $\delta_{2k}$ lie in $\ds_{ell}$, so $v_a$ gives an
injective map $\fv\hookrightarrow \ds_{ell}$.  This concludes
the proof.
\end{proof}
\vspace{.1cm}
\begin{rmk} \label{rmk:dsellvspls}
In \cite{Brown:depth3}, a Lie algebra called $\pls$ (for ``polar linearized shuffle'') is introduced, which is essentially equivalent to $\ds_{ell}$. It is also shown that $\fu$ embeds into $\pls$ (\cite{Brown:depth3}, Proposition 4.6) and, moreover, it is asked whether the equality $\fu=\pls$ holds. Proposition \ref{dsell}.(iii) implies that $\ds_{ell}$ is, in fact, much larger than $\fu$. More precisely, Enriquez (\cite{Enriquez:EllAss}, \S 7) has shown that $\fu$ lies in the kernel of the surjection $\grt_{ell} \rightarrow \grt$ from which it follows that the image $\gamma(\grt) \subset \widetilde{\grt}_{ell}$ of $\grt$ under the splitting $\gamma$ is disjoint from $\fu$. In particular, the Lie algebra $\fu$ cannot equal $\ds_{ell}$.
\end{rmk}
\subsection{The elliptic double shuffle relations} We can now
give the elliptic double shuffle property satisfied by the
reduced elliptic generating series $\overline{E}(\tau)$, in terms of moulds.
Let
$$\mathfrak{e}_m(\tau)=ma\bigl(\mathfrak{e}(\tau)\bigr)$$
and
$$\overline{E}_m(\tau)=ma\bigl(\overline{E}(\tau)\bigr)$$
\begin{thm}\label{ellipticds} (i) The mould $\mathfrak{e}_m(\tau)$
is $\Delta$-bialternal, that is, $\Delta^{-1}\bigl(\mathfrak{e}_m(\tau)\bigr)$
is a bialternal mould.
\vskip .1cm
\noindent (ii) The mould $\overline{E}_m(\tau)$ is $\Delta^*$-bisymmetral,
that is, $(\Delta^*)^{-1}\bigl(\overline{E}_m(\tau)\bigr)$ is bisymmetral.
\end{thm}
\begin{proof} (i) We have
$\mathfrak{e}(\tau)=r(\tau)\cdot a+\gamma(\phi_{\rm KZ})\cdot a+s(\tau)\cdot a$.
Let $\mathfrak{e}=\gamma(\phi_{\rm KZ})\cdot a$, so that
$\mathfrak{e}\in v_a\bigl(\gamma(\grt)\bigr)\otimes_\Q\overline\Ec$. By the 
proof of Theorem \ref{tensprod}, we have $r(\tau)\cdot a+s(\tau)\in V
\otimes_\Q \overline{\Ec}$ where $V=v_a(\fu')$.
Therefore, $\mathfrak{e}(\tau) \in \widetilde{\grt}_{ell}\otimes_\Q \overline{\Ec}$ by the definition of $\widetilde{\grt}_{ell}$, and since $v_a(\widetilde{\grt}_{ell}) \subset \ds_{ell}$ by Proposition \ref{dsell} (iii), we also have
$\mathfrak{e}(\tau)\in \ds_{ell}\otimes_\Q\overline\Ec$, which proves
the theorem thanks to \eqref{dselldef}.

\vspace{.2cm}
(ii) If the mould $\overline{E}_m(\tau)$ were equal to 
$\exp_{Dari}\bigl(\mathfrak{e}_m(\tau)\bigr)$, that is, if $\overline{E}(\tau)$ were equal
to $\exp_a\bigl(\mathfrak{e}(\tau)\bigr)$, this result would be immediate by
definition.  However, this is not quite the case.  Indeed, if it were true, it would
mean that as power series, $\overline{E}(\tau)$ would be equal to $\exp_a\bigl(\mathfrak{e}(\tau)\bigr)$, but this is impossible due to the fact that 
$\mathfrak{e}(\tau)=D(a)$
where $D=ch_{[\,,\,]}\bigl(r(\tau),\gamma(\phi_{\rm KZ}\bigr)\not\in
\Der'_0(\ff_2)$, whereas the 
diagram \eqref{diag0} only commutes for derivations lying in $\Der'_0(\ff_2)$.
In order to show that $\overline{E}_m(\tau)$ is nonetheless $(\Delta^*)^{-1}$-bisymmetral,
we will show that there exists a derivation $\delta\in \fu'\rtimes\gamma(\nmz^{\vee})$
such that the automorphism $\exp(\delta)$ coincides with $\exp(D)$ on $a$.  

Because $D\in \Der_0(\ff_2)$ and thus annihilates $[a,b]$, the automorphism 
$\exp(D)$ fixes $[a,b]$.  Define $\alpha$ to be the automorphism of $F_2(\Ec^{\rm geom}
\otimes_\Q \overline{\Zc})$ given by 
$$\alpha=\exp(D)\circ \exp(-2\pi i\tau\eps_0)=\exp\bigl(ch_{[\,,\,]}(D,
-2\pi i\tau\epsilon_0)\bigr).$$  
Then
$$\alpha(a)=\exp(D)\cdot a,\ \ \ \alpha(b)=\exp(D)\cdot b-2\pi i\tau\exp(D)\cdot a.$$
We also have
$$\alpha([a,b])=[\alpha(a),\alpha(b)]=
[\exp(D)\cdot a,\exp(D)\cdot b-2\pi i\tau \exp(D)\cdot a]$$
$$=[\exp(D)\cdot a,\exp(D)\cdot b]=[a,b].$$
Let $\delta$ be the derivation 
$$\delta:=\log(\alpha)=ch_{[\,,\,]}\bigl(D,-2\pi i\tau\eps_0\bigr).$$  
Then since $\alpha$ fixes $[a,b]$, $\delta$ annihilates $[a,b]$ and it has no linear term in $b$ since this is true for both $D$ and $\eps_0$. Therefore $\delta\in \Der_0(\ff_2)$.  Furthermore, the weight $1$ part of 
$\alpha(a)$ is $a$ and the weight $1$ part of $\alpha(b)$ is $b$, so $\delta$ 
is a strictly weight-increasing derivation, that is, there is no $\eps_0$ term in 
$\delta$, so $\delta\in \Der'_0(\ff_2)$. 
Finally, to see that $\delta\in\fu'\rtimes
\gamma(\nmz^{\vee})$, we write 
\begin{equation}\label{chterms}
\begin{aligned}
\delta&=ch_{[\,,\,]}\bigl(D,-2\pi i\tau\eps_0\bigr)\\
&=r'(\tau)+\gamma(\phi_{\rm KZ})+s(\tau)+{\frac{1}{2}}[r(\tau)+\gamma(\phi_{\rm KZ})+s(\tau),-2\pi i\tau\eps_0]+\cdots,
\end{aligned}
\end{equation}
where $r'(\tau)=r(\tau)-2\pi i\tau\eps_0\in \fu'$.
The terms $r'(\tau)$ and $\gamma(\phi_{\rm KZ})$ lie in $\fu'$ and $\gamma(\nmz^{\vee})$.
It was shown in Lemma \ref{semidir} that $\fu\rtimes\gamma(\nmz^{\vee})$ is
a semi-direct product, so all bracketed words lie in $\fu$, and in fact
in $\fu'$ since a bracketed word cannot be of degree $0$ (as a derivation).
Thus all the bracketed terms in the right-hand side of \eqref{chterms} 
lie in $\fu'$.
Therefore $\delta\in \fu'$, and so $v_a(\delta)=\delta(a)\in\ds_{ell}$ by 
(iii) of Proposition~\ref{dsell}.
Thus $ma\bigl(\delta(a)\bigr)$ is $\Delta$-bialternal,
so $ma\bigl(\exp(\delta)\cdot a\bigr)=ma\bigl(\alpha(a)\bigr)$ is 
$\Delta^*$-bisymmetral.  Since $\alpha(a)=\exp(D)\cdot a=\overline{E}(\tau)$, this completes the proof of
(ii).
\end{proof}

\vspace{.2cm}
We conjecture that the elliptic double shuffle relations 
form a complete set of algebraic relations between the $\overline E$-elliptic multizetas.  This statement really breaks down into
two statements, one concerning the arithmetic part $\overline{\Zc}$ of $\overline{\Ec}$ and the 
other the geometric part $\Ec^{\rm geom}=\U(\fu)^{\vee}$. 
We show in Proposition \ref{conjs} that indeed, the completeness follows from
two conjectures: the first one a standard conjecture from multizeta
theory, and the second a similar conjecture from elliptic multizeta
theory.  Due to the fact that it is much easier to work in the geometric
situation than the arithmetic situation (as there are no problems of
transcendence), we are actually able to prove that the elliptic
double shuffle relations are complete in depth 2, without any recourse 
to conjectures (see Proposition \ref{prop:conj2}).

The first conjecture is the standard conjecture that the double shuffle relations
suffice to generate all the algebraic relations satisfied by multiple zeta values 
\cite{IKZ}.  It can be stated as the vector space isomorphism
$$\nz\cong \ds^{\vee},\leqno{\hbox{\bf Conjecture 1:}}$$
or equivalently as $\overline{\Zc} \cong \U(\ds)^{\vee}$. Knowing the inclusions \eqref{inclusions} and also the existence of the injective
morphism $\grt\subset\ds$ (proved by Furusho in \cite{Furusho:Associators}), Conjecture
1 implies in particular that $\grt\cong \ds$.

The second conjecture essentially comes down to the elliptic analogue of this
statement, $\grt_{ell}\cong \ds_{ell}$.  Let us formulate it more precisely.
Recall that Enriquez showed that there is a canonical surjection
$\grt_{ell}\rightarrow\!\!\!\!\!\rightarrow\grt$.  
(Note that the terminology $\grt_{ell}$ used throughout this article corresponds to the Lie 
algebra denoted $\grt^{ell}_1$ by Enriquez, and our $\grt$ corresponds to his $\grt_1$.)  
The kernel of this surjection
is a normal Lie subalgebra denoted $\mathfrak{r}_{ell}\subset \grt_{ell}$.  Enriquez
shows that $\mathfrak{r}_{ell}$ contains $\mathfrak{sl}_2$ and the derivations $\delta_{2k}$
defined in \eqref{deltas}, and conjectures that $\mathfrak{r}_{ell}$ is actually generated
by these elements (\cite{Enriquez:EllAss}, \S 10.1).

Set $\grt^0_{ell}:=\grt_{ell}\cap \Der_0(\ff_2)$ and $\mathfrak{r}^0_{ell}=
\mathfrak{r}_{ell}\cap \Der_0(\ff_2)$. Then the semi-direct product structure on $\grt_{ell}$ 
yields the
analogous one on the restricted space: $\grt^0_{ell}\cong \mathfrak{r}^0_{ell}\rtimes
\gamma(\grt)$, and Enriquez' conjecture on $\mathfrak{r}_{ell}$ restricts to the conjecture
$\mathfrak{r}^0_{ell}\cong \fu$. Putting these together, we have the conjecture
$\grt^0_{ell}\cong \fu\rtimes\gamma(\grt)\subset \Der_0(\ff_2)$.  
Our second conjecture is the double shuffle analogue of this statement.
$$\fu \rtimes \gamma(\ds) \cong \ds_{ell}\subset \Der_0(\ff_2).\leqno{\hbox{\bf Conjecture 2:}}$$

\vspace{.2cm}
\begin{prop}\label{conjs} If Conjectures 1 and 2 are true, then 
$$\Ec^{\rm geom}\otimes_\Q\overline{\Zc}\cong \U(\ds_{ell})^\vee;$$
in other words, the elliptic double shuffle relations generate a complete family of 
algebraic relations between elliptic multizetas mod $2\pi i$.
\end{prop}

\begin{proof} 
By Conjecture 1, we would have 
$\overline\Zc\cong \U(\ds)^\vee$, so since $\Ec^{\rm geom}\cong \U(\fu)^\vee$
by Theorem \ref{duality}, we would have
$$\Ec^{\rm geom}\otimes_\Q\overline{\Zc}\cong \U(\fu)^\vee\otimes_\Q\U(\ds)^\vee.$$

As recalled in Step 2 of the proof of Theorem \ref{tensprod}, the underlying 
$\Q$-algebra of the graded dual Hopf algebra of the universal enveloping algebra of a
semi-direct product of graded Lie algebras is isomorphic to the tensor product of the
duals of the two Lie algebras, so since $\gamma:\ds\rightarrow\gamma(\ds)$ is
a $\Q$-isomorphism, we have the isomorphism of $\Q$-algebras
$\U(\fu)^\vee\otimes_\Q\U(\ds)^\vee\cong \U\bigl(\fu\rtimes\gamma(\ds)\bigr)^\vee$.
Thus, using conjecture 1 and conjecture 2, we find
\begin{equation}\label{isos}
\Ec^{\rm geom}\otimes_\Q \overline{\Zc}\cong \U(\fu)^\vee\otimes_\Q\U(\ds)^\vee\cong 
\U\bigl(\fu\rtimes\gamma(\ds)\bigr)^\vee\cong \U(\ds_{ell})^\vee.
\end{equation}
This isomorphism means that the algebraic relations satisfied by the elements of
$\overline{\Ec}[2\pi i\tau]$ are generated by the elliptic double shuffle relations. 
\end{proof}

\vspace{.2cm}
\noindent {\it Explicit elliptic double shuffle relations.} 
Let us take a closer look at what the $\Delta$-bialternality properties
are. The first property is that $\mathfrak{e}_m(\tau)$ is $\Delta$-alternal, 
that is, that $\Delta^{-1}(\mathfrak{e}_m(\tau))$ is alternal. However, $\Delta$ 
trivially preserves alternality, so this is equivalent to saying that 
$\mathfrak{e}_m(\tau)$ is alternal, that is, that for each $r>1$,
$$\sum_{u\in sh\bigl((u_1,\ldots,u_k),(u_{k+1},\ldots,u_r)\bigr)} 
\mathfrak{e}_m(\tau)(u)=0\leqno(EDS.1)$$
for $1\le k\le [r/2]$.  This condition is equivalent to the statement
that the power series $\mathfrak{e}(\tau)$ is a Lie series.  

The new relations on $\mathfrak{e}_m(\tau)$ are the second set, which 
say that up to adding on a constant-valued mould, the swap of the mould 
$\Delta^{-1}\bigl(\mathfrak{e}_m(\tau)\bigr)$ is also alternal, 
where the swap-operator is defined in \eqref{swapmould}.
This alternality is given by the equalities for $r>1$
$$\sum_{v\in sh\bigl((v_1,\ldots,v_k),(v_{k+1},\ldots,v_r)\bigr)} swap\bigl(\Delta^{-1}\mathfrak{e}_m(\tau)\bigr)(v)=0\leqno(EDS.2)$$
for $1\le k\le [r/2]$.

The swapped mould is given explicitly by
$$swap\bigl(\Delta^{-1}\mathfrak{e}_m(\tau)\bigr)
=\frac{1}{v_1(v_1-v_2)\cdots(v_{r-1}-v_r)v_r}\, \mathfrak{e}_m(\tau)(v_r,v_{r-1}-v_r, \ldots,v_1-v_2).$$

Thus the alternality conditions in (EDS.2) are all sums of rational 
functions with denominators that are products of terms of the form 
$v_i$ and $(v_i-v_j)$, which sum to zero. Therefore, by multiplying through by 
the common denominator
\begin{equation*}
v_1\cdots v_r\prod_{i>j}(v_i-v_j),
\end{equation*}
the second elliptic shuffle equation can be expressed as a family of 
polynomial conditions on the mould $swap(\mathfrak{e}_m(\tau))$.  

\vspace{.2cm}
\noindent {\it Elliptic double shuffle relations in depth 2.} Let us
work this out explicitly in depth 2.
The usual alternality condition reduces to 
$$\mathfrak{e}_m(\tau)(u_1,u_2)+\mathfrak{e}_m(\tau)(u_2,u_1)=0.\leqno(EDS.1\hbox{-}{\rm depth}\ 2)$$
The swap alternality condition reads
\begin{equation*}
\frac{1}{v_1(v_1-v_2)v_2}swap(\mathfrak{e}_m(\tau))(v_1,v_2)+
\frac{1}{v_1(v_2-v_1)v_2}swap(\mathfrak{e}_m(\tau))(v_2,v_1)=0,
\end{equation*}
which, clearing denominators, reduces simply to
$$swap(\mathfrak{e}_m(\tau))(v_1,v_2)-swap(\mathfrak{e}_m(\tau))(v_2,v_1)=0.$$
Since $swap(\mathfrak{e}_m(\tau))(v_1,v_2)=e_m(v_2, v_1-v_2)$, this is given by the relation
$$\mathfrak{e}_m(\tau)(v_2,v_1-v_2)=\mathfrak{e}_m(\tau)(v_1,v_2-v_1)$$
directly on $\mathfrak{e}_m(\tau)$.  Applying the depth 2 swap operator from
$\overline{ARI}$ to $ARI$ (given by $v_1\mapsto u_1+u_2$, $v_2\mapsto u_1$),
we transform this relation into
$$\mathfrak{e}_m(\tau)(u_1,u_2)=\mathfrak{e}_m(\tau)(u_1+u_2,-u_2).$$
Finally, $\mathfrak{e}_m(\tau)$ is of odd degree, so by the depth 2 version
of (EDS.1), we have
$\mathfrak{e}_m(\tau)(-u_2,-u_1)=\mathfrak{e}_m(\tau)(u_1,u_2)$, which gives
$$\mathfrak{e}_m(\tau)(u_1,u_2)=\mathfrak{e}_m(\tau)(u_2,-u_1-u_2).\leqno(EDS.2\hbox{-}{\rm depth}\ 2)$$

Note that this is nothing other than
$\mathfrak{e}_m(\tau)(u_1,u_2)=push\bigl(\mathfrak{e}_m(\tau)\bigr)(u_1,u_2)$
where the push-operator is defined in \eqref{pushmould}.  Thus in depth 2,
the $\Delta$-bialternality conditions correspond to alternality and
push-invariance of $\mathfrak{e}_m(\tau)$ (which in turn correspond
to the fact that $\mathfrak{e}(\tau)$ is a Lie series that is push-invariant
in depth 2 in the sense of power series, as in \eqref{pushop}).  This
simple reformulation is special to depth 2; the $\Delta$-bialternal
property does not lend itself so easily to a direct expression as
a property of $\mathfrak{e}(\tau)$ in higher depths.

We end this subsection by showing that the conjecture that the
$\Delta$-bialternal relations are sufficient holds in depth 2.

\begin{prop}\label{prop:conj2} The relations (EDS.1) and (EDS.2) in odd degrees
are the only relations satisfied by $\mathfrak{e}_m(\tau)$ in depth 2.
\end{prop}

\begin{proof} We can prove this result without recourse to any conjectures,
essentially because depth 2 is too small to contain any of the arithmetic part 
of $\mathfrak{e}_m(\tau)$ (we qualify this statement below), and the geometric part $V=v_a(\fu)$ is well-understood in depth two.  
We know that $\mathfrak{e}(\tau)\in \ds_{ell}\subset \ff_2^{push}$.
The graded dimensions of $\ff_2$ in depth 2 are given by
\begin{equation}\label{dimliepush}
dim(\ff_2^{\rm push})_n^2=\left\lfloor\frac{n-5}{6}\right\rfloor+1.
\end{equation}
Now the depth two part of $\ds_{ell} \supset V \rtimes \gamma(\ds)$ is contained in the depth two part of $V$, since $\gamma(\ds)$ is of depth $\geq 3$. Thus
\begin{equation}\label{dimensions}
{\rm dim}\,\bigl(\ds_{ell}\bigr)_n^2={\rm dim}\,V_n^2=\begin{cases}\left\lfloor\frac{n-5}{6}\right\rfloor+1&\hbox{if $n$ is odd $\ge 5$}\\
0&\hbox{otherwise}.  \end{cases}
\end{equation}
Indeed, the last equality follows from the fact that in depth 2, $V$ is spanned by the 
$[\eps_{2j},\eps_{2k}](a)$ with $j<k$, $j,k\ne 1$, which are all of
odd weight, and the fact that, as shown in \cite{Pollack:Thesis}, the only relations between these $\left\lfloor\frac{n-3}{4}\right\rfloor$ brackets come from period polynomials,
whose number is given by $\left\lfloor\frac{n-7}{4}\right\rfloor-\left\lfloor\frac{n-5}{6}\right\rfloor$.
Thus $V^2=\ds_{ell}^2=(\ff_2^{push})^2$, so the Lie relation (EDS.1) and
the push-invariance relation (EDS.2) suffice to characterize elements
of $\ds_{ell}$ in depth 2.  \end{proof}

\noindent {\bf Depth 2 elements of $\ds_{ell}$ in low weights:}

\vspace{.2cm}
$\bullet$ in weight $5$, 
$$ma\bigl([\eps_0,\eps_4](a)\bigr)=2u_1^3+3u_1^2u_2-3u_1u_2^2-2u_2^3.$$

\vspace{.2cm}
$\bullet$ in weight $7$, 
$$ma\bigl([\eps_0,\eps_6](a)\bigr)=2u_1^5+5u_1^4u_2+2u_1^3u_2^2-2u_1^2u_2^3-5u_1u_2^4-2u_2^5.$$

\vspace{.2cm}
$\bullet$ in weight $9$,
$$ma\bigl([\eps_0,\eps_8](a)\bigr)=2u_1^7+7u_1^6u_2+9u_1^5u_2^2+5u_1^4u_2^3-5u_1^3u_2^4-9u_1^2u_2^5-7u_1u_2^6 2u_2^7.$$

\vspace{.2cm}
$\bullet$ in weight $11$, 
$$ma\bigl([\eps_0,\eps_{10}](a)\bigr)=8u_1^9+36u_1^8u_2+74u_1^7u_2^2+91u_1^6u_2^3+41u_1^5u_2^4-41u_1^4u_2^5$$ $$-91u_1^3u_2^6-74u_1^2u_2^7-36u_1u_2^8-8u_2^9$$
$$\ \ ma\bigl([\eps_4,\eps_6](a)\bigr)=-2u_1^7u_2^2-7u_1^6u_2^3-5u_1^5u_2^4
+5u_1^4u_2^5+7u_1^3u_2^6+2u_1^2u_2^7.$$

\subsection{The elliptic associator and the push-neutrality relations mod $2\pi i$}

\begin{dfn} Let $\mathfrak{a}$ be the power series
with coefficients in $\overline\Zc$ given by
$$\mathfrak{a}=\frac{1}{2\pi i}\log(A)\ {\rm mod}\ 2\pi i=
\log(\overline{A'})= \overline\Phi_{\rm KZ}(t_{01},t_{12})
t_{01}\overline\Phi_{\rm KZ}(t_{01},t_{12})^{-1},$$
and let $\mathfrak{a}(\tau)=g(\tau)\cdot \mathfrak{a}$.
\end{dfn} 

The coefficients of $\mathfrak{a}(\tau)$ generate the $\Q$-algebra 
$\overline{\Ac}$ of
$\overline{A}$-EMZs. In this paragraph we will consider certain
relations satisfied by the coefficients of $\mathfrak{a}(\tau)$,
different from the linearized elliptic double shuffle relations satisfied
by $\mathfrak{e}(\tau)$.  The first family of relations on the coefficients of
$\mathfrak{a}(\tau)$ is the usual family of {\it alternality} relations, 
but the second is the family of {\it push-neutrality} relations.  These 
relations are related (mod $2\pi i$) to the {\it Fay-shuffle relations} 
introduced in \cite{Matthes:Edzv}, and studied explicitly in depth 2.  
We show that in depth 2, the push-neutrality relations are identical to
the Fay-shuffle relations.  We also show that even in depth 2 and mod 
$2\pi i$, the alternality and push-neutrality relations are strictly weaker 
than the linearized elliptic double shuffle relations.

We will give our relations in terms of mould theory (but see
Corollary \ref{pushneut} for a translation into power series terms at the end).
For this we recall the $push$ and $dar$-operators defined in
\eqref{pushmould} and \eqref{darmould}. We will say that a mould $B$ is 
{\it push-neutral} if 
\begin{equation}\label{pushneutrality}
B(u_1,\ldots,u_r)+push(B)(u_1,\ldots,u_r)+\cdots+push^r(B)(u_1,\ldots,u_r)=0
\end{equation}
for all $r\ge 1$, where $push$ denotes the push-operator on moulds defined in 
\eqref{pushmould}.

\begin{thm}\label{FS} Let $\mathfrak{a}_m(\tau)=ma\bigl(\mathfrak{a}(\tau)\bigr)$.
Then $\mathfrak{a}_m(\tau)$ is alternal and $dar^{-1}\bigl(\mathfrak{a}_m(\tau)\bigr)$ is push-neutral in depth $r>1$.
\end{thm}

\begin{proof} Recall the derivation $arat$ defined in \eqref{aratdef}.  For
any $P\in ARI$, set
\begin{equation}\label{Daritdef}
\Darit(P)=dar\circ arat\bigl(\Delta^{-1}(P)\bigr)\circ dar^{-1}.
\end{equation}
It is shown in \cite{Schneps:Emzv}, Lemma 3.1.2, that the map
\begin{align}
\Der_0(\ff_2)&\hookrightarrow \Der(ARI_{lu})\notag\\
D&\mapsto \Darit\bigl(ma(v_a(D))\bigr)\label{Daritmap}
\end{align}
is an injective Lie morphism, so that we have
\begin{equation}\label{mader}
ma\bigl(D(f)\bigr)=\Darit\bigl(ma(v_a(D))\bigr)\cdot ma(f).
\end{equation}

Let $\mathfrak{a}_m=ma(\mathfrak{a})$, $\mathfrak{a}_m(\tau)=
ma\bigl(\mathfrak{a}(\tau)\bigr)$, and $r_m(\tau)=ma\bigl(r_a(\tau)\bigr)$. 
Under the map \eqref{Daritmap}, 
we have $r(\tau)\mapsto \Darit\bigl(r_m(\tau)\bigr)$, so
$$ma\bigl(r(\tau)\cdot \mathfrak{a}\bigr)=\Darit\bigl(r_m(\tau)\bigr)
\cdot \mathfrak{a}_m.$$
Since
\begin{equation}\label{gtaum}
\mathfrak{a}(\tau)=g(\tau)\cdot \mathfrak{a}=\sum_{n\ge 0} \frac{1}{n!}r(\tau)^n\cdot \mathfrak{a},
\end{equation}
we have
\begin{equation}\label{gtaum2}
\mathfrak{a}_m(\tau) = \sum_{n\ge 0}\frac{1}{n!}\Darit\bigl(r_m(\tau)\bigr)^n\cdot \mathfrak{a}_m.
\end{equation}

Let $\overline\sigma$ denote the automorphism of 
$\ff_2$ defined in \S 3.2. We have
$$\mathfrak{a}=\overline\sigma(t_{01}).$$
Recall from \S 3.2 that $\overline\sigma=\gamma(\phi_{\rm KZ})$,
where $\phi_{\rm KZ}=\log_a\bigl(\overline\Phi_{\rm KZ}\bigr)$.

The derivation $\gamma(\phi_{\rm KZ})$ lies in $\Der_0(\ff_2)$, so
$\gamma(\phi_{\rm KZ})\cdot t_{01}\in \ff_2$; thus $\mathfrak{a}$ is
a Lie series.  Since $r(\tau)\in \Der_0(\ff_2)$, we have
$r(\tau)^n\cdot \mathfrak{a}\in \ff_2$ for all $n\ge 1$, so by
\eqref{gtaum}, $\mathfrak{a}(\tau)=
g(\tau)\cdot \mathfrak{a}\in\ff_2$, which means that
$\mathfrak{a}_m(\tau)$ is alternal.  This settles the first property
of $\mathfrak{a}_m(\tau)$ stated in the theorem.

Let us consider the second property.
Since $\gamma(\phi_{\rm KZ})\in \Der_0(\ff_2)$, it annihilates $t_{12}$.  
Therefore, setting
$t'_{01}=t_{01}+\frac{1}{2}t_{12}$, we have
\begin{equation}\label{equ1}
\mathfrak{a}=\gamma(\phi_{\rm KZ})\cdot t_{01}=\gamma(\phi_{\rm KZ})\cdot t'_{01}.
\end{equation}

Set $T'_{01}=ma(t'_{01})$, and set
$$\mathfrak{z}=ma\Bigl(v_a\bigl(\gamma(\phi_{\rm KZ})\bigr)\bigr)=
ma\bigl(\gamma_a(\phi_{\rm KZ})\bigr).$$
Then by \eqref{mader}, the equality \eqref{equ1}
translates into moulds as
$$\mathfrak{a}_m=\Darit(\mathfrak{z})\cdot T'_{01}.$$

To complete the proof of the second property, we will use the following lemma, whose proof is deferred to the final subsection of this paper.

\begin{lem}\label{aratpushneutral} Let $A\in ARI$. If $A$ is push-neutral, then $arat(P)\cdot A$ is push-neutral
for all $P\in ARI$.  If $dar^{-1}A$ is push-neutral, then
$dar^{-1}\cdot \Darit(P)\cdot A$ is push-neutral for all $P\in ARI$.
\end{lem}

It is easy to see that if $A$ is a push-invariant mould, then
$dar^{-1}A$ is push-neutral, since
\begin{equation}
\begin{aligned}
&\sum_{s=0}^rpush^s\bigl(dar^{-1}(A)\bigr)(u_1,\ldots,u_r)\\
&=\biggl(\frac{1}{u_1\cdots u_r}+\frac{1}{u_2\cdots u_0}+\cdots
+\frac{1}{u_0u_1\cdots u_{r-1}}\biggr)A(u_1,\ldots,u_r)\\
&=\biggl(\frac{u_0+u_1+\cdots+u_r}{u_0u_1\cdots u_r}\biggr)A(u_1,\ldots,u_r)\\
&=0,
\end{aligned}
\end{equation}
where $u_0=-u_1-\cdots-u_r$. By Proposition \ref{t01pushneutral} below, $dar^{-1}T'_{01}$ is push-neutral
and by Lemma \ref{aratpushneutral}, so is 
$$dar^{-1}\mathfrak{a}_m=dar^{-1}\cdot \Darit(\mathfrak{z})\cdot T'_{01}.$$

To show that $dar^{-1}\mathfrak{a}_m(\tau)$ is push-neutral we use the
same lemma again. Since $dar^{-1}\mathfrak{a}_m$ is push-neutral,
so is $dar^{-1}\cdot \Darit\bigl(r_m(\tau)\bigr)\cdot \mathfrak{a}_m$,
and then successively, so is
$dar^{-1}\cdot \Darit\bigl(r_m(\tau)\bigr)^n\cdot \mathfrak{a}_m$
for all $n\ge 1$. Thus $dar^{-1}\mathfrak{a}_m(\tau)$ is push-neutral
by \eqref{gtaum2}. This proves the theorem.\end{proof}

The following proposition was used in the proof of Theorem \ref{FS}.
\begin{prop}\label{t01pushneutral}
The mould
\begin{equation}
ma([t'_{01},a])=-\sum_{n=2}^{\infty}\frac{B_n}{n!}ma([\ad^n(b)(a),a])
\end{equation}
is push-neutral.
\end{prop}
\begin{proof}
It is enough to show the push-neutrality of $f_n:=ma([\ad^n(b)(a),a])$ for all $n \geq 2$ separately. Using the definition of $ma$ (cf. Section \ref{ssec:moulds}), we see that
\begin{equation}
ma(\ad^n(b)(a))=-\sum_{k=1}^n(-1)^{n-k}\binom{n-1}{k-1}u_k \in \Q[u_1,\ldots,u_n].
\end{equation}
Now in depth $n$, the operator $\ad(a)$ on $\Q\langle\!\langle C\rangle\!\rangle$ corresponds to multiplication by $-(u_1+\ldots+u_n)$. Consequently,
\begin{align}
ma([\ad^n(b)(a),a])&=-ma([a,\ad^n(b)(a)])\notag\\
&=-(u_1+\ldots+u_n)\sum_{k=1}^n(-1)^{n-k}\binom{n-1}{k-1}u_k\notag\\
&=-\sum_{j,k=1}^n(-1)^{n-k}\binom{n-1}{k-1}u_ju_k.
\end{align}
On the other hand, by the definition of the push-operator \eqref{pushmould}, we have $push(f_n)=-\sum_{j,k=1}^n(-1)^{n-k}\binom{n-1}{k-1}u_{j+1}u_{k+1}$, where the indices are to be taken mod $n$ (so that $u_{k+n}=u_k$). Using the elementary fact that $\sum_{k=1}^n(-1)^{n-k}\binom{n-1}{k-1}=0$ for $n \geq 2$, it is now clear that
\begin{equation}
\sum_{i=0}^{n-1}push^i(f_n)=0,
\end{equation}
that is, $f_n$ is push-neutral for all $n \geq 2$, as was to be shown.
\end{proof}

\vspace{.2cm}
We end this subsection by studying these relations more explicitly in
depth 2 and comparing them with the elliptic double shuffle relations
on $\mathfrak{e}_m(\tau)$.  The alternality relation is of course
the same:
$$\mathfrak{a}_m(\tau)(u_1,u_2)+\mathfrak{a}_m(\tau)(u_2,u_1)=0.\leqno(FS.1)$$
The push-neutrality relation in depth 2 is given by
$$\frac{1}{u_1u_2}\mathfrak{a}_m(\tau)(u_1,u_2)+\frac{1}{u_2u_0}\mathfrak{a}_m(\tau)(u_2,u_0)+\frac{1}{u_0u_1}\mathfrak{a}_m(\tau)(u_0,u_1)=0
\leqno(FS.2)$$
where $u_0=-u_1-u_2$.  Multiplying by the common denominator 
$u_0u_1u_2$ yields the polynomial relation
$$u_0\mathfrak{a}_m(\tau)(u_1,u_2)+u_1\mathfrak{a}_m(\tau)(u_2,u_0)+
u_2\mathfrak{a}_m(\tau)(u_0,u_1)=0.$$

It was shown in \cite{Matthes:Edzv} Theorem 3.11, that the dimension of the space of 
polynomials in $u_1$ and $u_2$ of odd
degree $d$ satisfying (FS.1) and (FS.2) is given by
$\left\lfloor\frac{d}{3}\right\rfloor+1$.
In terms of the weight $n=d+2$ of the corresponding polynomials 
in $\ff_2$, this is
$$\left\lfloor\frac{n-2}{3}\right\rfloor+1.$$
In weight $5$, for example, there are two independent such polynomials:
$$u_1^2u_2-u_1u_2^2\ \ \ {\rm and}\ \ \ u_1^3-u_2^3.$$
In weight $7$, there are again two independent polynomials, given by
$$u_1^4u_2-u_1u_2^4\ \ \ {\rm and}\ \ \ u_1^5+u_1^3u_2^2-u_1^2u_2^3-u_2^5.$$
In weight $9$, the space is three-dimensional, given by

\vspace{.2cm}
$u_1^7-2u_1^4u_2^3+2u_1^3u_2^4-u_2^7$

$u_1^6u_2-u_1u_2^6$

$u_1^5u_2^2+u_1^4u_2^3-u_1^3u_2^4-u_1^2u_2^5$.
\vskip .1cm
\noindent Finally, we work out the case of weight $11$, where the dimension is four:
\vskip .1cm

$u_1^9+3u_1^5u_2^4-u_1^4u_2^5-u_2^9$

$u_1^8u_2-u_1u_2^8$

$u_1^7u_2^2-u_1^5u_2^4+u_1^4u_2^5-u_1^2u_2^7$

$u_1^6u_2^3+u_1^5u_2^4-u_1^4u_2^5-u_1^3u_2^6$

\vspace{.2cm}
Observe that these dimensions are significantly bigger than those
given by the elliptic double shuffle equations (EDS.1) and (EDS.2)
in depth 2.  This is explained by the fact that the vector space generated
by the coefficients of $\mathfrak{a}_m(\tau)$ in a given weight and
depth is not equal to the one generated by the analogous coefficients
of $\mathfrak{e}_m(\tau)$, or in terms of the algebras, that
while $\overline{A}[2\pi i\tau]=\overline{\Ec}[2\pi i\tau]$ by virtue
of Theorem \ref{ellipticds}, the $\Q$-algebras $\overline{\Ec}$ and $\overline{\Ac}$
are quite different and do not even have the same graded dimensions.

Under the conjecture $\overline{\Zc}\cong \U(\grt)^\vee$, the $\Q$-algebra
$\overline{\Ec}$ is isomorphic to $\U(\grt_{ell})^\vee$, and thus inherits
a natural bigrading dual to that of $\grt_{ell}$.  Together with products
of elements of $\overline{\Ec}$ of smaller depth and weight (including $\Gc_0$),
the coefficients of $\mathfrak{e}_m(\tau)$ in a given weight $n$ and depth 
$d$ span the bigraded part $\overline{\Ec}_n^d$, whereas those of 
$\mathfrak{a}_m(\tau)$ do not.

For example, in weight $5$ and depth $2$, the coefficients of $\mathfrak{e}_m(\tau)$ generate the 1-dimensional space $\langle 2\Gc_{0,4}+\Gc_0\Gc_4\rangle$.
The bigraded subspace $\overline{\Ec}_5^2$ is spanned by 
$\Gc_2^2$, $\Gc_0\Gc_4$ and $\Gc_{0,4}$, but it is also spanned by the two 
products $\Gc_2^2$ and $\Gc_0\Gc_4$ and the single coefficient 
$2\Gc_{0,4}+\Gc_0\Gc_4$ of $\mathfrak{e}_m(\tau)$ in weight $5$ and depth $2$.

The weight $5$, depth $2$ coefficients of $\mathfrak{a}_m(\tau)$, however,
do not lie in $\overline{\Ec}_5^2$. They span the 2-dimensional
subspace $\langle -\frac{1}{12}\Gc_0\Gc_2+\frac{3}{2}\Gc_0\Gc_4
+3\Gc_{0,4}-\frac{1}{360}\Gc_0^2+\frac{1}{2}\Gc_2^2,
\frac{1}{240}\Gc_0^2-2\Gc_{0,4}-\Gc_0\Gc_4\rangle$ of $\overline{\Ec}$.

\vspace{.3cm}
Finally, we give a statement about alternality and push-neutrality relations
on $\mathfrak{a}_m(\tau)$.

\begin{coro}\label{pushneut} The power series
$\mathfrak{A}=[a,\mathfrak{a}(\tau)]$ is push-neutral in the sense that,
if $\mathfrak{A}^r$ denotes the depth $r$ part of $\mathfrak{A}$ for
$r>1$, then
$$A^r+push(A^r)+\cdots+push^r(A^r)=0$$
where $push$ denotes the push-operator on power series defined in 
\eqref{pushop}.
\end{coro}

\begin{proof} By Theorem \ref{FS}, the mould $dar^{-1}\mathfrak{a}_m(\tau)$
is push-neutral.  Consider the operator
$$-\Delta(A)(u_1,\ldots,u_r)=u_1\cdots u_r(-u_1-\ldots-u_r)A(u_1,\ldots,u_r).$$
Since the factor $u_1\ldots u_r(-u_1-\ldots-u_r)$ is push-invariant,
the mould $-\Delta(A)$ is push-neutral if $A$ is.  Therefore
in particular $-\Delta\bigl(dar^{-1}\mathfrak{a}_m(\tau)\bigr)$
is push-neutral. However, this mould is given by
\begin{align*}-\Delta\bigl(dar^{-1}\mathfrak{a}_m(\tau)\bigr)(u_1,\ldots,u_r)&=
-(u_1+\cdots+u_r)\,\mathfrak{a}_m(\tau)(u_1,\ldots,u_r)\\
&=ma\bigl([a,\mathfrak{a}(\tau)]\bigr)(u_1,\ldots,u_r),
\end{align*}
where the last equality is a standard identity (see Appendix A
of \cite{Racinet:Thesis} or (3.3.1) of \cite{Schneps:ARI}). Therefore
the mould $ma([a,\mathfrak{a}(\tau)])$ is a push-neutral mould, 
that is, $[a,\mathfrak{a}(\tau)]$ is push-neutral as a power series.
\end{proof}

\subsection{Proof of Lemma \ref{aratpushneutral}.}
In order to prove this lemma, we need to have recourse to the complete
formula for the action of $arat$. We first recall \'Ecalle's formula
for $arit$ (cf.~\cite{Ecalle:Flexion} or \cite{Schneps:ARI}), which is given as
$$\bigl(arit(P)\cdot A\bigr)(w)=\sum_{{\scriptscriptstyle\substack{w=abc\\c\ne \emptyset}}}
A(a\lceil c)P(b)-\sum_{{\scriptscriptstyle\substack{w=abc\\a\ne \emptyset}}}A(a\rceil c) P(b),$$
where if the word $u=(u_1,\ldots,u_r)$ is decomposed into three
chunks as $u=abc$, $a=(u_1,\ldots,u_i)$, $b=(u_{i+1},\ldots,u_{i+j})$,
$c=(u_{i+j+1},\ldots,u_r)$, then we use \'Ecalle's notation
$$a\rceil = (u_1,\ldots,u_{i-1},u_i+u_{i+1}+\cdots+u_{i+j})$$
$$\lceil c= (u_{i+1}+\cdots+u_{i+j+1},u_{i+j+2},\ldots,u_r).$$

Moreover 
$$\ad(P)\cdot A=mu(P,A)-mu(A,P)$$ 
where $mu$ is the mould multiplication defined in \eqref{mumould};
these correspond precisely to the `missing' terms $a=\emptyset$ and 
$c=\emptyset$, so that $arat(P)\cdot A$ actually has the simpler expression
\begin{equation}\label{aratPA}
\bigl(arat(P)\cdot A\bigr)(w)=\sum_{w=abc} \bigl(A(a\lceil c)P(b)-A(a\rceil c) P(b)\bigr).
\end{equation}
Now let $A$ be push-neutral, and let $P\in ARI$.  We need to show that 
\eqref{aratPA} is push-neutral.  In fact we will show that the two 
terms 
\begin{equation}    \label{amat}
\sum_{w=abc} A(a\lceil c)P(b)\ \ {\rm and}\ \
\sum_{w=abc} A(a\rceil c) P(b)
\end{equation}
of \eqref{aratPA} are separately push-neutral. 

Because the push-neutrality relations take place in fixed depth, we may assume 
that $A$ is concentrated in depth $s$ and $P$ in depth $t$, with $s+t=r$.
We will prove the push-neutrality of the first term in (\ref{amat}); the proof for the
second term is completely analogous.

Therefore the decompositions $w=abc$ we need to consider are those of the form
$$w=abc=(u_1,\ldots,u_i)(u_{i+1},\ldots,u_{i+t})(u_{i+t+1},\ldots,u_r),$$
and we can rewrite the first term of (\ref{amat}) as
$$\sum_{i=0}^{r-t} A(u_1,\ldots,u_i,u_{i+1}+\cdots+u_{i+t+1},u_{i+t+2},
\ldots,u_r)P(u_{i+1},\ldots,u_{i+t}).$$
The $k$-th power of the push-operator acts by $u_i\mapsto u_{i-k}$,
with indices considered modulo $(r+1)$.  The push-neutrality condition thus reads
$$\sum_{k=0}^r\sum_{i=0}^{r-t} A(u_{1-k},\ldots,u_{i-1-k},u_{i-k},u_{i+1-k}+
\cdots+u_{i+t+1-k},u_{i+t+2-k}, \ldots,u_{r-k})$$
$$\qquad\qquad\qquad\qquad\qquad\cdot P(u_{i+1-k},\ldots,u_{i+t-k})
=0.$$
We will show that the coefficients of each term $P(u_{m+1},\ldots,u_{m+t})$ sum to zero 
due to the push-neutrality of $A$.  In fact it is enough to show that the coefficients of 
$P(u_1,\ldots,u_t)$ sum to zero, as all the other terms are obtained from this one by 
applying powers of the push-operator.  

The terms containing $P(u_1,\ldots,u_t)$ are those for which the index $k=i$, so that 
$k\in \{0,\ldots,r-t=s\}$, and we must show that the sum
$$\sum_{k=0}^s A(u_{r-k+2},\ldots,u_r,u_0,u_1+\cdots+u_{t+1},u_{t+2}, \ldots,u_{r-k+1})$$
vanishes, where $u_0=-u_1-\cdots-u_r$ and we have shifted some of the indices modulo 
$(r+1)$ in order to make them positive.  Note now that
$$u_1+\cdots+u_{t+1}=-u_0-u_{t+2}-\cdots-u_r.$$  
As a result the last sum runs over the $(s+1)$ cyclic permutations of 
$u_{t+2},\ldots,u_r,u_0$ and $-u_{t+2}-\cdots-u_r-u_0$, so it is equal to the sum 
over the push$_s$-orbit of just one term, say the one with $k=s$, that is, to
$$\sum_{k=0}^s push_s^k\bigl(A(u_{t+2},\ldots,u_r,u_0)\bigr).$$ 
However, this indeed vanishes, since $A$ is push-neutral. This concludes the
proof of Lemma \ref{aratpushneutral}.\hfill{$\square$}

\vspace{1cm}

\bibliographystyle{abbrv}
\bibliography{Bibliography}

\end{document}